\documentclass[reqno, 11pt, a4paper]{amsart}

\usepackage[english]{babel}
\usepackage{amsmath}
\usepackage{amssymb}
\usepackage{enumerate}
\usepackage{commath}
\usepackage{mdwlist}
\usepackage{mathtools}
\usepackage{ mathrsfs }
\usepackage[arrow, matrix, curve]{xy}
\usepackage[a4paper]{geometry}
\usepackage{fullpage}
\usepackage{framed}
\usepackage{tikz}
\usetikzlibrary{arrows}
\usepackage{mdframed}
\usepackage{lipsum}
\usepackage{mathtools}
\usepackage{verbatim}
\usepackage{bbold}
\usepackage{eufrak}
\usepackage[normalem]{ulem}

\usepackage{wrapfig}
\usepackage{graphicx}
\usepackage{tabularx}
\usepackage{verbatim}
\usepackage{amsthm}
\usepackage{commath}
\usepackage{amsfonts}
\usepackage{mathrsfs}
\usepackage{amsopn}
\usepackage{stackrel}
\usepackage{manfnt}
\usepackage{bbm}
\usepackage{fancyhdr}
\usepackage{centernot}
\usepackage{stmaryrd}
\usepackage{comment}
\usepackage{cancel}
\usepackage{hyperref}
\usepackage{kantlipsum}
\usepackage{pgfplots}

\def\cdf(#1)(#2)(#3){1-0.5*(1+(erf((#1-#2)/(#3*sqrt(2)))))}%

\allowdisplaybreaks

\theoremstyle{definition}
\newtheorem{defn}{Definition}[section]

\newmdtheoremenv{Definition}[defn]{Definition}

\newtheorem{remark}[defn]{Remark}

\newtheorem*{corol*}{Corollary}
\newtheorem*{remark*}{Bemerkung}

\newtheorem*{exmp*}{Example}

\newtheorem*{subalg*}{\textsf{Sub-Algorithmus}}

\theoremstyle{plain}
\newtheorem{thm}[defn]{Theorem}
\newmdtheoremenv{theo}[defn]{Theorem}
\newmdtheoremenv{Lemma}[defn]{Lemma}
\newmdtheoremenv{Korollar}[defn]{Corollary}
\newtheorem*{thm*}{Theorem}
\newtheorem{lemma}[defn]{Lemma}

\newtheorem{prop}[defn]{Proposition}
\newtheorem{corol}[defn]{Corollary}

\newtheorem*{thmen*}{Theorem}

\newtheorem*{corolen*}{Corollary}

\numberwithin{equation}{section}

\newcommand{\R}{\ensuremath{\mathbb{R}}}

\newcommand{\N}{\ensuremath{\mathbb{N}}}

\newcommand{\E}{\ensuremath{\mathbb{E}}}

\newcommand{\ind}{\ensuremath{\mathbbm{1}}}
\newcommand{\defeq}{\ensuremath{\vcentcolon}=}
\newcommand{\eps}{\epsilon}

\newcommand{\PP}{\ensuremath{\mathbb{P}}}

\newcommand{\ddd}{\ensuremath{\text{d}}}

\newcommand{\vertiii}[1]{\vert\kern-0.25ex\vert\kern-0.25ex\vert #1 
    \vert\kern-0.25ex\vert\kern-0.25ex\vert}

\begin{document}

\title{The stochastic Fisher-KPP Equation with seed bank and on/off branching coalescing Brownian motion}

\author{Jochen Blath}
\address{Technische Universit\"at Berlin}
\curraddr{Strasse des 17. Juni 136, 10623 Berlin}
\email{blath@math.tu-berlin.de}
\thanks{}

\author{Matthias Hammer}
\address{Technische Universit\"at Berlin}
\curraddr{Strasse des 17. Juni 136, 10623 Berlin}
\email{hammer@math.tu-berlin.de}
\thanks{}

\author{Florian Nie*}
\address{Technische Universit\"at Berlin}
\curraddr{Strasse des 17. Juni 136, 10623 Berlin}
\email{nie@math.tu-berlin.de*}
\thanks{}


\keywords{Fisher-Kolmogoroff-Petrovski-Piscounov, traveling wave, duality, dormancy, seed bank, on/off branching Brownian motion, delay spde\\
*\textit{Corresponding author e-mail address}: nie@math.tu-berlin.de
}

\date{\today}

\begin{abstract}
We introduce a new class of stochastic partial differential equations (SPDEs) with seed bank modeling the spread of a beneficial allele in a spatial population where individuals may switch between an active and a dormant state. Incorporating dormancy and the resulting seed bank leads to a two-type coupled system of equations with migration between both states. We first discuss existence and uniqueness of seed bank SPDEs and provide an equivalent delay representation that allows a clear interpretation of the age structure in the seed bank component. The delay representation will also be crucial in the proofs. Further, we show that the seed bank SPDEs give rise to an interesting class of ``on/off”-moment duals. In particular, in the special case of the F-KPP Equation with seed bank, the moment dual is given by an ``on/off-branching Brownian motion”. This system differs from a classical branching Brownian motion in the sense that independently for all individuals, motion and branching may be ``switched off" for an exponential amount of time after which they get ``switched on" again. On/off branching Brownian motion shows qualitatively different behaviour to classical branching Brownian motion and is an interesting object for study in itself. Here, as an application of our duality, we show that the spread of a beneficial allele, which in the classical F-KPP Equation, started from a Heaviside intial condition, evolves as a pulled traveling wave with speed $\sqrt{2}$, is slowed down significantly in the corresponding seed bank F-KPP model. In fact, by computing bounds on the position of the rightmost particle in the dual on/off branching Brownian motion, we obtain an upper bound for the speed of propagation of the beneficial allele given by $\sqrt{\sqrt{5}-1}\approx 1.111$ under unit switching rates. This shows that seed banks will indeed slow down fitness waves and preserve genetic variability, in line with intuitive reasoning from population genetics and ecology.
\end{abstract}

\maketitle

\section{Introduction and main results}
\label{sec:intro}

\subsection{Motivation}

One of the most fundamental models in spatial population genetics and ecology, describing the spread of a beneficial allele subject to directional selection, was introduced by Fisher in \cite{F37}. Denoting by $p(t,x)\in [0,1]$ the frequency of the advantageous allele at time $t\geq 0$ and spatial position $x \in \R$, and assuming diffusive migration of individuals (described by the Laplacian), Fisher considered the partial differential equation 
\begin{equation}
\label{eq:F-KPP}
    \partial_t p(t,x) = \frac{\Delta}{2} p(t,x) -p(t,x)^2 + p(t,x).
\end{equation}
The same system was independently investigated around the same time by Kolmogorov, Petrovsky, and Piscounov in \cite{KPP37}, and thus the above PDE is now commonly known (and abbreviated) as F-KPP Equation, see e.g.\ \cite{bovier_2016} for a recent overview. It is well known that there exists a so called \textit{travelling wave solution} with speed $\sqrt{2}$  meaning that there exists a function $w$ such that
\begin{align}
    p(t,x)=w(x-\sqrt{2}t)
\end{align}
solves \eqref{eq:F-KPP}. Much finer results about the asymptotic behaviour of the wave-speed and the shape of the function $w$ are known (see e.g.\ \cite{B83}, \cite{LS87}, \cite{R13}), and the F-KPP Equation and its extensions with different noise terms are still an active field of research (see e.g.\ \cite{K17}, \cite{M19}). A very interesting feature of the F-KPP Equation and a main reason for the amenability of its analysis is given by the fact that the solution to \eqref{eq:F-KPP} is {\em dual} to branching Brownian motion (BBM), as was shown by McKean \cite{McK75} (and earlier by Ikeda, Nagasawa and Watanabe \cite{INW69}). Indeed, starting in a (reversed) Heaviside initial condition given by $p(0,\cdot):= {\bf 1}_{]-\infty, 0]}$, we have the probabilistic representation
\begin{align}
    p(t,x)=1-\PP_0(R_t \leq x),
\end{align}
where $(R_t)_ {t\geq 0}$ is the position of the rightmost particle of a (binary) branching Brownian motion with branching rate 1, started with a single particle in $0$. Bramson \cite{B78} then also showed that the rightmost particle of this system thus governs the asymptotic wave-speed of the original equation via the equality
\begin{align} \label{eq: AsymptoticWaveSpeedF-KPP}
    \lim_{t \to \infty} \frac{R_t}{t}= \sqrt{2}.
\end{align}
Since the days of Fisher, mathematical modeling in population genetics has expanded rapidly, and many additional ``evolutionary forces'' have been incorporated into the above model. For example, one may include mutations between alleles and a ``Wright-Fisher noise'' as a result of random reproduction, leading to the system 
\begin{align}
\label{eq:F-KPP_with_noise_and_forces}
\partial_t p(t,x) = & \frac{\Delta}{2} p(t,x) + m_1(1- p(t, x)) - m_2 p(t, x)\notag 
+sp(t,x)(1-p(t,x)) \\
&\quad + \sqrt{\nu p(t,x)(1-p(t,x))} \dot W(t,x).
\end{align}
Here $m_1\geq 0$ and $m_2\geq 0$ are the mutation rates to and from the beneficial allele, $s\geq 0$ denotes the strength of the selective advantage of the beneficial allele, $\nu \geq 0$ governs the variance of the reproductive mechanism and $W=(W(t,x))_{t \geq 0, x \in \R}$ denotes a Gaussian white noise process. The Wright-Fisher noise term is the standard null-model of population genetics, in the non-spatial setting corresponding to an ancestry governed by the Kingman-coalescent \cite{K82}. A justification for its use in population genetics can be found in \cite{M19}. \\
\noindent From a biological point of view  one may think of a one-dimensional habitat modeled by $\R$ on which two types (or species) compete for limited resources. The Heaviside initial condition (induced perhaps by some initial spatial barrier separating the two interacting types) admits a detailed analysis of the impact of the selective advantage of the beneficial type on its propagation in space (see e.g.\ \cite{SN97}).

\medskip

\noindent Recently, an additional evolutionary mechanism has drawn considerable attention in population genetics. Indeed, {\em dormancy}, and, as a result, {\em seed banks}, are both ubiquitous in microbial species as well as crucial for an understanding of their evolution and ecology (see e.g.\ \cite{LJ11}, \cite{SL18}). Corresponding discrete-space population genetic models have recently been studied in \cite{g20} and non-spatial models, where dormancy and resuscitation are modeled in the form of classical migration between an active and an inactive state, have been derived and investigated in \cite{BEGCKW15} and \cite{BGCKW16} (these papers also provide biological background and motivation). There, the population follows a two-dimensional ``seed bank diffusion'', given by the system of SDEs
\begin{align}
\label{eq:seed_bank_diffusion}
    \ddd p(t) &= c(q(t)-p(t)) + \sqrt{p(t)(1-p(t))} \ddd B(t), \notag \\
    \ddd q(t)  &= c'(p(t)-q(t))
\end{align}
where $p$ describes the frequency of the allele under consideration in the active population, and $q$ its frequency in the dormant population. The constants $c,c'>0$ represent the switching rates between the active and dormant states, respectively, and $(B(t))_{t\geq 0}$ is a standard Brownian motion. Though reminiscent of Wright's two island model (cf.\ \cite{BBGCWB18+}), the system above exhibits quite unique features. For example, it is dual to an ``on/off''-coalescent (instead of the Kingman coalescent), in which lines may be turned on and off with independent exponential rates given by $c$ and $c'$. Lines which are turned off are prevented from coalescences. Note that this structure also appears in the context of meta-population models from ecology, see \cite{LM15}. It can be shown that this new ``seed bank coalescent'' does not come down from infinity and exhibits qualitatively prolonged times to the most recent common ancestor \cite{BGCKW16}. A further interesting feature is that the above system exhibits a long-term memory, which can be well understood in a delay SDE reformulation obtained in (\cite[Prop.\ 1.4]{BBGCWB18+}).
Assume starting frequencies $p_0=x\in [0,1], q_0=y \in [0,1]$ and for simplicity $c=c'=1$. Then, the solution to \eqref{eq:seed_bank_diffusion} is a.s.\ equal to the unique strong solution of the stochastic delay differential equations
\begin{align*}
\partial_t p(t)  &= \Big( y e^{-t} + \int_0^t e^{-(t-s)}  p(s)  {\rm{d}}s  - p(t) \Big)  {\rm{d}} t + \sqrt{p(t)(1-p(t)}{\rm{d}}B_t, \notag \\
\partial_t q(t)  &= \Big( -y e^{-t} - \int_{0}^{t} e^{-(t-s)}  p(s) {\rm{d}}s + p(t) \Big) {\rm{d}} t\notag 
\end{align*}
with the same initial condition. The result rests on the fact that there is no noise in the second component and can be proved by a integration-by-parts argument. The second component is now just a deterministic function of the first.

\medskip

\noindent It appears natural to incorporate the above seed bank components into a F-KPP framework in order to analyse the combined effects of seed banks, space and directional selection. We will thus investigate systems of type
\begin{align*} 
    \partial_t p(t,x) &= c (q(t,x)-p(t,x))+\frac{\Delta}{2} p(t,x) + s  (p(t,x)-p^2(t,x))  \nonumber\\
    &\quad+ m_1(1- p(t,x))- m_2 p(t,x) +  \sqrt{ \nu p(t,x)(1-p(t,x))} \dot W (t,x), \nonumber\\
    \partial_t q(t,x) &= c'(p(t,x) -q(t,x))
\end{align*}
where $c,c'\geq 0$ are the switching rates between active and dormant states, $s\geq 0$ is the selection parameter, $\nu\geq 0$ the reproduction parameter and $m_1,m_2 \geq 0$ are the mutation parameters. One may view this as a continuous stepping stone model (cf.\ \cite{S88}) with seed bank. Due to technical reasons, which will become clear in Section \ref{sec:Uniqueness_and_duality}, it is actually advantageous for us to consider in the following the process 
$$(u,v)\defeq (1-p,1-q) $$
satisfying the system
\begin{align} \label{eq:SteppingStoneWithSeedbank}
    \partial_t u(t,x) &= c (v(t,x)-u(t,x))+\frac{\Delta}{2} u(t,x) + s  (u^2(t,x)-u(t,x))  \nonumber\\
    &\quad- m_1 u(t,x)+ m_2 (1-u(t,x)) +  \sqrt{ \nu u(t,x)(1-u(t,x)} \dot W (t,x), \nonumber\\
    \partial_t v(t,x) &= c'(u(t,x) -v(t,x))
\end{align}
instead. We expect to see the on/off mechanism of \eqref{eq:seed_bank_diffusion} emerge also in the dual of the above system. In particular, in the F-KPP Equation with seed bank 
given by 
\begin{align} 
\partial_t u(t,x) = & c (v(t,x)-u(t,x))+ \frac{\Delta}{2} u(t,x)  -su(t,x)(1-u(t,x)), \notag\\
    \partial_t v(t,x) =& c'(u(t,x) -v(t,x))\label{eq:F-KPP:with_seed_bank}
\end{align}
we expect to obtain an ``on/off branching Brownian motion'' with switching rates $c,c'$ and branching rate $s$ as a moment dual.
 Further we aim to derive a delay representation for the above SPDE and hope to get at least partial information about the wave speed of a potential traveling wave solution. Intuition from ecology suggests that the spread of the beneficial allele should be slowed down due to the presence of a seed bank. However, we also expect new technical problems, since the second component $v(t,x)$ comes without the Laplacian, so that all initial roughness of $v_0$ will be retained for all times, preventing jointly continuous solutions.

\subsection{Main results}
\noindent In this section, we  summarize the main results of this paper.
We begin by showing that our Equation \eqref{eq:SteppingStoneWithSeedbank} is well-defined, i.e.\ we establish weak existence, uniqueness and boundedness of solutions. This is done via the following theorems:
\begin{thm} \label{thm: existence}
The SPDE given by Equation \eqref{eq:SteppingStoneWithSeedbank}
for $s,c,c',m_1,m_2, \nu\geq 0$ with initial conditions $(u_0,v_0) \in B(\R,[0,1]) \times B(\R,[0,1])$ has a weak solution $(u,v)$ (in the sense of Definition \ref{def: weaksolution} below) with paths taking values in $C(]0,\infty[,C(\R,[0,1])) \times C([0,\infty [, B(\R,[0,1]))$.
\end{thm}
\noindent Here, for Banach spaces $X$ and $Y$ we denoted by $B(X,Y)$ the space of bounded, measurable functions on $X$ taking values in $Y$ and by $C(X,Y)$ the space of continuous functions on $X$ taking values in $Y$. We usually suppress the dependence on the image space whenever our functions are real-valued and equip both spaces with the topology of locally uniform convergence.

\begin{thm} \label{thm:Uniqueness in law}
Under the conditions of Theorem \ref{thm: existence}, the SPDE \eqref{eq:SteppingStoneWithSeedbank} exhibits uniqueness in law
on $C(]0,\infty[,C(\R,[0,1])) \times C([0,\infty [, B(\R,[0,1]))$.
\end{thm}

\noindent Note that it turns out that the absence of a Laplacian in  the second equation in \eqref{eq:SteppingStoneWithSeedbank} gives rise to technical difficulties regarding existence and uniqueness. However, as in the seed bank diffusion case, a reformulation as a stochastic partial delay differential equation is possible allowing one to tackle these issues. To our knowledge, this is a new application of a delay representation in this context, and a detailed explanation of this approach can be found in Section \ref{sec:Existence}.
\begin{prop} \label{prop: Delay}
The Equation \eqref{eq:SteppingStoneWithSeedbank} is equivalent to the Stochastic Partial Delay Differential Equation (SPDDE)
    \begin{align}
    \partial_t u(t,x) &= c \left(e^{-c't}v_0(x)+ c' \int_0^t e^{-c'(t-s)} u(s,x) \, \ddd s-u(t,x)\right)+\frac{\Delta}{2} u(t,x)-m_1 u(t,x)\nonumber\\
    &\quad +m_2 (1-u(t,x)) - s u(t,x)(1-u(t,x))  +\sqrt{\nu (1-u(t,x))u(t,x)} \dot W(t,x),\nonumber\\
    \partial_t  v(t,x) &= c' \left(u(t,x)-e^{-c't}v_0(x)-e^{-c't} c' \int_0^t e^{c's} u(s,x) \, \ddd s \right)\label{eq:SteppingStoneWithSeedbankDelay}
    \end{align}
    in the sense that under the same initial conditions solutions of \eqref{eq:SteppingStoneWithSeedbank} are also solutions of \eqref{eq:SteppingStoneWithSeedbankDelay} and vice versa.
\end{prop}
\begin{remark}
Proposition \ref{prop: Delay} gives rise to an elegant interpretation of the delay term. It shows that the type of any “infinitesimal” resuscitated individual is determined by the active population present an exponentially  distributed  time  ago (with a cutoff at time 0), which  the  individual spent dormant in the seed bank (cf. Proposition 1.4. in \cite{BBGCWB18+}).
\end{remark}
\noindent Another major tool needed for deriving the 
uniqueness result is the powerful duality technique, i.e.\ we prove a moment duality with an ``on/off branching coalescing Brownian motion" (with killing) which as in \cite{AT00} we define slightly informally as follows. For a rigorous construction, we refer the reader to the killing and repasting procedure of Ikeda, Nagasawa and Watanabe (cf.\ \cite{INW68a}, \cite{INW68b}, \cite{INW69}) or \cite{A98}. Note also that the introduction of the on/off-mechanism will lead, as in the on/off-coalescent case in \cite{BGCKW16}, to an extension of the state space allowing each particle to carry an active or dormant marker.
\begin{defn} \label{defn:dualprocess}
We denote by $M=(M_t)_{t \geq 0}$ an on/off branching coalescing Brownian motion with killing taking values in $\bigcup_{k \in \N_0} \left(\R \times \lbrace \boldsymbol{a},\boldsymbol{d}\rbrace \right)^k$ starting at $M_0= ((x_1,\sigma_1), \cdots, (x_n,\sigma_n ))\in\left(\R \times \lbrace \boldsymbol{a}, \boldsymbol{d}\rbrace\right)^n$ for some $n \in \N$. Here the marker $\boldsymbol{a}$ (resp. $\boldsymbol{d}$) means that the corresponding particle is active (resp. dormant).
The process evolves according to the following rules:
\begin{itemize}
    \item Active particles, i.e.\ particles with the marker $\boldsymbol{a}$, move in $\R$ according to independent Brownian motions, die at rate $m_2$ and branch into two active particles at rate $s$.
    \item Pairs of active particles coalesce according to the following mechanism:
    \begin{itemize}
        \item We define for each pair of particles labelled $(\alpha, \beta)$ their intersection local time $L^{\alpha, \beta}=(L^{\alpha , \beta}_t)_{t \geq 0}$  as the local time of $M^\alpha-M^\beta$ at $0$ which we assume to only increase whenever both particles carry the marker $\boldsymbol{a}$.
        \item Whenever the intersection local time exceeds the value of an independent exponential clock with rate $\nu/2$, the two involved particles coalesce into a single particle.
    \end{itemize}
    \item Independently, each active particle switches to a dormant state at rate $c$ by switching its marker from $\boldsymbol{a}$ to $\boldsymbol{d}$.
    \item Dormant particles do not move, branch, die or coalesce.
    \item Independently, each dormant particle switches to an active state at rate $c'$ by switching its marker from $\boldsymbol{d}$ to $\boldsymbol{a}$.
\end{itemize}
Moreover, denote by $I=(I_t)_{t\geq 0}$ and $J=(J_t)_{t \geq 0}$ the (time dependent) index set of active and dormant particles of $M$, respectively, and let $N_t$ be the random number of particles at time $t\geq 0$ so that $M_t=(M^1_t, \ldots , M^{N_t}_t)$. For example, if for $t\geq 0$ we have 
$$M_t=((M^1_t,\boldsymbol{a}), (M^2_t,\boldsymbol{d}),(M^3_t,\boldsymbol{a}),  (M^4_t, \boldsymbol{a})),$$
then 
$$I_t=\lbrace1,3,4 \rbrace,\, J_t=\lbrace 2 \rbrace,\, N_t=4.$$
\end{defn}
\begin{remark} \label{rem:specialcases}
For future use we highlight the following special cases of the process $M$. They admit the same mechanisms as described in Definition \ref{defn:dualprocess} except for those where we set the rate to $0$:
\begin{itemize}
    \item $m_1=m_2=0$: $M$ is called an on/off branching coalescing Brownian motion (without killing) or on/off BCBM.
    \item $m_1=m_2=\nu=0$: $M$ is called an on/off branching Brownian motion (without killing) or on/off BBM.
    \item $m_1=m_2=s=0$: $M$ is called an on/off coalescing Brownian motion (without killing) or on/off CBM.
    \item $m_1=m_2=\nu=s=0$: $M$ is called an on/off Brownian motion (without killing) or on/off BM.
\end{itemize}
\end{remark}
\noindent 
We have the following moment duality for the process $M$ of Definition \ref{defn:dualprocess} which uniquely determines the law of the solution of the system \eqref{eq:SteppingStoneWithSeedbank}.
\begin{thm} \label{thm:duality}
 Let $(u,v)$ be a solution to the system \eqref{eq:SteppingStoneWithSeedbank} with initial conditions $u_0,v_0 \in B (\R,[0,1])$. Then we have for any initial state $M_0 =((x_1, \sigma_1), \ldots , (x_n, \sigma_n)) \in \left(\R \times \lbrace \boldsymbol{a}, \boldsymbol{d} \rbrace\right)^n$, $n\in \N$, and for any $t \geq 0$
 \begin{align*}
     \E\left[ \prod_{\beta \in I_0} u(t,M^\beta_0)  \prod_{\gamma \in J_0} v(t, M^\gamma_0) \right]&= \E\left[ \prod_{\beta \in I_t} u_0(M^\beta_t)  \prod_{\gamma \in J_t} v_0( M^\gamma_t) e^{-m_1 \int_0^t \abs{I_s} \, \ddd s}\right].
 \end{align*}
\end{thm}

\noindent Finally, as an application of the preceding results we consider the special case without mutation and noise. This is the F-KPP Equation with seed bank, i.e.\ Equation \eqref{eq:F-KPP:with_seed_bank}. In this scenario the duality relation takes the following form.
\begin{corol} \label{corol:dualityF-KPP}
Let $( u, v)$ be the solution to Equation \eqref{eq:F-KPP:with_seed_bank} with initial condition $u_0=v_0=\ind_{[0,\infty[}$. Then the dual process $M$ is an on/off BBM (see Remark \ref{rem:specialcases}). Moreover, if we start $M$ from a single active particle, the duality relation is given by
\begin{align*}
     u(t,x) = \PP_{(0,\boldsymbol{a})}\left(\max_{\beta \in I_t\cup J_t} M_t^\beta \leq x \right).
\end{align*}
\end{corol}
\noindent Similar to Bramson's result on the asymptotic speed of the rightmost particle (cf.\ Equation \eqref{eq: AsymptoticWaveSpeedF-KPP}), by means of the preceding duality  we prove an upper bound on the propagation speed of the beneficial allele. 
\begin{thm} \label{thm:Wavespeed_c=1}
For $s=c=c'= 1$, $u_0=v_0=\ind_{[0,\infty[}$ and any $\lambda \geq  \sqrt{\sqrt{5}-1}$ we have that
\begin{align*}
    \lim_{t \to \infty} 1-u(t,\lambda t)=0.
\end{align*}
Moreover, it holds almost surely that
\begin{align*}
    \limsup_{t\to \infty} \frac{R_t}{t} \leq \sqrt{\sqrt{5}-1}
\end{align*}
where $R_t=\sup_{\alpha \in I_t \cup J_t} M_t^\alpha$ is the position of the rightmost particle of the on/off BBM started from a single active particle at $0$. 
\end{thm}
\noindent In particular, we see that the propagation speed of the beneficial allele and the rightmost particle of the dual process is significantly reduced to at least $\approx 1.11$ compared to the previous speed of $\sqrt{2}$ in the case of the classical F-KPP Equation \eqref{eq:F-KPP}. A more general statement highlighting the exact dependence of $\lambda$ on the switching parameters $c$ and $c'$ is given later in Proposition \ref{prop:NewBoundWaveSpeed} and Corollary \ref{corol: propagationspeed}.
\begin{figure}[ht]
	\centering
  \scalebox{0.9}[0.5]{\includegraphics[width=0.75\textwidth]{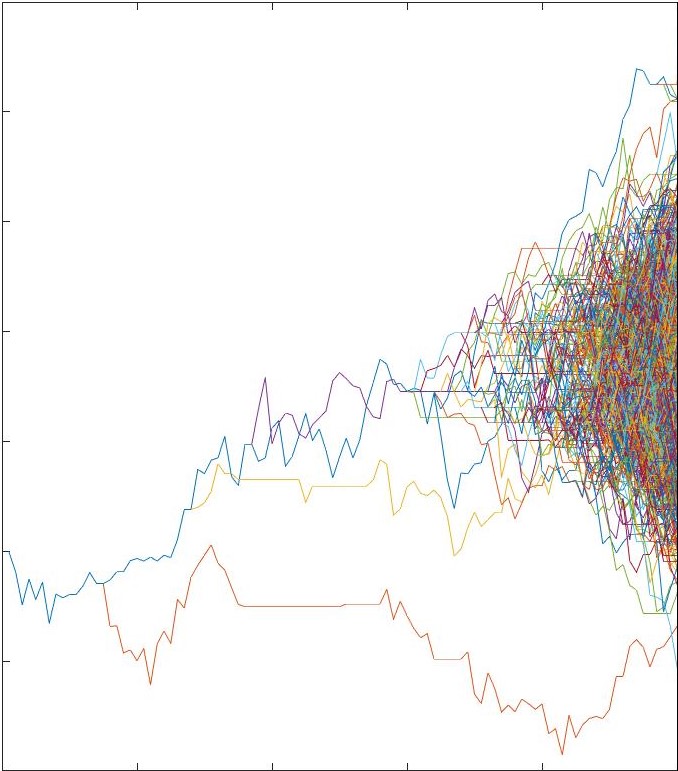}}
	\caption{Simulation of an on/off-BBM. Horizontal lines appear whenever motion of a particle is switched off.}
	\label{fig:OnOffBBM}
\end{figure}
\subsection{Outline of paper}

In Section \ref{sec:Existence}, we first state results concerning (weak) existence of solutions of our class of SPDEs from \eqref{eq:SteppingStoneWithSeedbank} and 
then prove the 
equivalent characterization of solutions in terms of
the delay representation \eqref{eq:SteppingStoneWithSeedbankDelay}. In Section \ref{sec:Uniqueness_and_duality}, we establish uniqueness (in law) of the solutions to \eqref{eq:SteppingStoneWithSeedbank} and show duality to on/off BCBM with killing. Then, in Section \ref{sec:F-KPP}, we investigate the special case of the F-KPP Equation with dormancy and show that the beneficial allele spreads at reduced speed in comparison with the corresponding classical F-KPP Equation (when started in Heaviside initial conditions). Finally, in Section \ref{sec:ProofForSection2} we provide outlines for the proofs of the results from Section \ref{sec:Existence}.

\section{Weak existence for a class of Stochastic Partial Differential Equations}
\label{sec:Existence}
\noindent In this section we provide a proof for Theorem \ref{thm: existence}. We begin by establishing strong existence and uniqueness for general systems of SPDEs with Lipschitz coefficients and use these results to obtain weak existence for systems with non-Lipschitz diffusion coefficients under some additional regularity assumptions. Finally, we show that Equation \eqref{eq:SteppingStoneWithSeedbank} fits into the previously established framework. In order to increase the readability of this section, we postpone the rather technical yet standard proofs of most theorems to Section \ref{sec:ProofForSection2}. \\
\noindent 
We begin with the definition of the white noise process which is crucial to the introduction of SPDEs.
\begin{defn}
A (space-time) white noise $W$ on $\R \times[0,\infty[$ is a zero-mean Gaussian process indexed by Borel subsets of $\R \times[0,\infty[$ with finite measure such that
$$\E[W(A)W(B)]=\lambda(A \cap B)$$
where $\lambda$ denotes the Lebesgue measure on $\R \times[0,\infty[$. If a set $A \in \mathcal{B}(\R \times[0,\infty[)$ is of the form $A=C\times [0,t]$ with $C \in \mathcal{B}(\R)$ we write $W_t(C)=W(A)$.
\end{defn}

\noindent We are now in a position to introduce the general setting of this section.
\begin{defn}
Denote by
\begin{align*}
   & b\colon  [0,\infty [ \times \R \times \R^{2} \to \R,
   & \tilde b \colon  [0,\infty [ \times \R \times \R^{2} \to \R
   \end{align*}
and 
\begin{align*}
    \sigma \colon [0,\infty [ \times \R \times \R^2 \to \R
\end{align*}
measurable maps. Then we consider the system of SPDEs
\begin{align} \label{eq: SPDE_general_shortnotation}
    \partial_t u(t,x) &= \frac{\Delta}{2} u(t,x)+b(t,x,u(t,x),v(t,x))  +\sigma(t,x,u(t,x),v(t,x)) \dot W(t,x), \nonumber\\
    \partial_t v(t,x) &= \tilde b(t,x,u(t,x),v(t,x))
\end{align}
with bounded initial conditions $u_0,v_0\in B(\R)$,
where $W$ is a $1$-dimensional white noise process.
\end{defn}
\noindent
The equation is to be interpreted in the usual analytically weak sense (cf. \cite{S94}), as follows:

\begin{defn} \label{def: weaksolution}
 Let $u_0,v_0 \in B(\R)$ 
and consider a 
random field
 $(u,v) =\allowdisplaybreaks (u(t,x),v(t,x))_{t\ge 0, x\in \R}$. 
\begin{itemize}
    \item
We say that $((u,v),W,\Omega , \mathcal{F},(\mathcal{F}_t)_{t \geq 0} , \PP)$ is a weak solution (in the stochastic sense) to Equation \eqref{eq: SPDE_general_shortnotation} with initial conditions $(u_0,v_0)$ if for each $\phi \in C^\infty_c (\R)$, almost surely it holds for all $t\ge0$ and $y\in\R$ that
        \begin{align} \label{spdeschwartz}
        \int_\R u(t,x) \phi(x) \, \ddd x &= \int_\R u_0(x) \phi(x) \, \ddd x +  \int_0^t \int_\R  u(s,x) \frac{\Delta}{2} \phi(x)  \, \ddd x \, \ddd s\\
        &\quad +   \int_0^t \int_\R b(s,x,u(s,x),v(s,x)) \phi(x) \, \ddd x  \,\ddd s \nonumber\\
        &\quad+  \int_0^t\int_\R \sigma(s,x,u(s,x), v(s,x)) \phi(x) \, W(\ddd s, \ddd x),\nonumber\\
\label{spde_v}         v(t,y)    &= v_0(y)  + \int_0^t \tilde b(s,y,u(s,y),v(s,y)) \, \ddd s
        \end{align}
and both $(u,v)$ and $W$ are adapted to $(\mathcal{F}_t)_{t \geq 0}$.
        We usually suppress the dependence of weak solutions on the underlying probability space and white noise process.
        \item We say that $(u,v)$ is a strong solution (in the stochastic sense) to Equation \eqref{eq: SPDE_general_shortnotation} if on some given probability space $(\Omega , \mathcal{F}, \PP)$ with some white noise process $W$, the process $(u,v)$ is adapted to the canonical filtration $(\mathcal{F}_t)_{t \geq 0}$ of $W$ and for each $\phi \in C^\infty_c (\R)$ satisfies almost surely 
\eqref{spdeschwartz}-\eqref{spde_v} for all $t \ge 0$ and $y\in\R$. 
\item  
For $p\ge2$, we say that a solution $(u,v)$ is $L^p$-bounded if 
\begin{align*}
\|(u,v)\|_{T,p}\defeq\sup_{0\leq t\leq T} \sup_{x \in \R} \E\left[\left(\vert u(t,x) \vert +\vert v(t,x) \vert \right)^p\right]^{1/p} < \infty
\end{align*}
for each $T>0$. 
\end{itemize}
\end{defn}

\noindent The solutions $(u,v)$ we are going to construct will have paths in $C(]0,\infty[,C(\R))\times C([0,\infty[,B_{\text{loc}}(\R))$. Here, we denote by $C(\R)$ resp.\ $B_{\text{loc}}(\R)$ the space of continuous resp.\ locally bounded measurable functions on $\R$. The spaces are endowed with the topology of locally uniform convergence. 
Note that for the random field $u$, this means equivalently that $u$ is jointly continuous on $]0,\infty[\times\R$. Since we allow for non-continuous (e.g.\ Heaviside) initial conditions $u_0$ and $v_0$, 
we have to restrict the path space for $u$ by excluding $t=0$. For the same reason and due to the absence of the Laplacian in \eqref{spde_v}, we cannot expect continuity of $v$ in the spatial variable $y$.

\noindent We start by establishing, 
for solutions with the above path properties, an equivalent mild representation involving the Gaussian heat kernel
$$G(t,x,y)=\frac{1}{\sqrt{2\pi t}}e^{\frac{(x-y)^2}{2t}}$$
which is the fundamental solution of the classical heat equation.

\begin{prop} \label{thm:SIE_Representation}
Let $u_0,v_0 \in B(\R) $ and assume that for all $T>0$, the linear growth condition
\begin{align} \label{eq: lineargrowth}
    |\tilde b(t,x,u,v)|+|b(t,x,u,v)| + |\sigma(t,x,u,v)  | 
   \leq C_T (1+|u|+|v|)
\end{align}
holds for every $(t,x,u,v)\in [0,T]\times \R \times \R \times \R$. 

\noindent Let $(u, v)$ be an adapted 
and $L^2$-bounded process with paths taking values in $C(]0,\infty[, C(\R))\times C([0,\infty[, B_{\mathrm{loc}}(\R))$.
 Then $(u,v)$ is a solution of Equation \eqref{eq: SPDE_general_shortnotation} in the sense of Definition \ref{def: weaksolution} iff $(u,v)$ satisfies the following Stochastic Integral Equation (SIE):
For each $t>0$ and $y\in\R$, almost surely it holds 
\begin{align} \label{SIE}
    u(t,y) &=  \int_\R   u_0(x) G (t,x,y)\, \ddd x  + \int_0^t\int_\R  b(s,x,u(s,x),v(s,x)) G(t-s,x,y) \ddd x\,\ddd s  \nonumber  \\
    &\quad + \int_0^t\int_\R  \sigma(s,x,u(s,x),v(s,x)) G(t-s,x,y) \,  W(\ddd s, \ddd x),
\end{align}
and almost surely Equation \eqref{spde_v} holds for all $t\ge0$ and $y\in\R$.

\end{prop}

\noindent Using the mild formulation of the equation, the next step is to show strong existence and uniqueness 
by a standard Picard iteration scheme. For this we need to impose the usual Lipschitz assumptions.

\begin{thm} \label{Lipschitzexistence}

Assume that for all  $T>0$, we have  the linear growth condition \eqref{eq: lineargrowth} and the following Lipschitz condition:
\begin{align} \label{eq:Lipschitzcondition}
    |\tilde b(t,x,u,v)-\tilde b(t,x,\tilde u,\tilde v)|+|b(t,x,u,v)-b(t,x,\tilde u,\tilde v)| + | \sigma(t,x,u,v) -\sigma(t,x,\tilde u,\tilde v)  | \nonumber\\
   \leq L_T (|u-\tilde u|+|v-\tilde v|)
\end{align}
for every $(t,x) \in [0,T]\times \R $ and $(u,v),(\tilde u, \tilde v) \in \R^2$.\\
Then for $u_0, v_0\in B(\R)$, Equation \eqref{eq: SPDE_general_shortnotation}
has a unique strong $L^2$-bounded solution $(u,v)$ with paths taking values in $C(]0,\infty[,C(\R))\times C([0,\infty[,B_{\mathrm{loc}}(\R))$.
Moreover, this solution is $L^p$-bounded for each $p\ge2$.

\end{thm}

\begin{remark}\label{rem_continuity}
Although the paths of $u$ are not continuous at $t=0$ for non-continuous initial conditions $u_0$, \textit{Step 2} in the proof of Theorem \ref{Lipschitzexistence} will in fact show that the process
\begin{align*}
    u(t,x)-\int_\R G(t,x,y) u_0(y)\, \ddd y
\end{align*}
has always paths in $C([0,\infty[,C(\R))$ and thus, in particular, is locally bounded in $(t,x)\in[0,\infty[\times \R$.
\end{remark}

\noindent Our next goal is to establish conditions under which we can ensure that the solutions to our SPDE stay 
in $[0,1]$. 

\begin{thm} \label{comparison}

Assume that the conditions of Theorem \ref{Lipschitzexistence} are satisfied. In addition, suppose that 
$b$ and $\tilde b$ are even Lipschitz continuous jointly\footnote{This is stronger than the Lipschitz condition \eqref{eq:Lipschitzcondition} which only requires that the bound holds in $(u,v)$.} in $(x,u,v)$ and satisfy the inequalities
\begin{align}
        b(t,x,0,v)&\geq 0\nonumber \text{ for all } (t,x,v)\in [0,\infty[\times \R \times \R ,\notag\\
    b(t,x,1,v)&\leq 0 \nonumber \text{ for all } (t,x,v)\in [0,\infty[\times \R \times \R,\notag\\
    \tilde b(t,x,u,0)&\geq 0\nonumber \text{ for all } (t,x,u)\in [0,\infty[\times \R \times \R,\notag\\
    \tilde b(t,x,u,1)&\leq 0  \text{ for all } (t,x,u)\in [0,\infty[\times \R \times \R \label{eq:inequality_b}.
\end{align}
Finally, assume that $\sigma$ is a function of $(t,x,u)$ alone, Lipschitz continuous jointly in $(x,u)$ and satisfies 
\begin{align} \label{conditionspositive}
    \sigma (t,x,0)&=0\text{ for all } (t,x)\in [0,\infty[\times \R ,\notag\\
    \sigma(t,x,1)&=0\text{ for all } (t,x)\in [0,\infty[\times \R .
\end{align}
For initial conditions $u_0,v_0 \in B(\R , [0,1])$,
let $(u,v)\in C(]0,\infty[,C(\R))\times C([0,\infty[,B_{\mathrm{loc}}(\R))$ be the unique strong $L^2$-bounded solution to Equation \eqref{eq: SPDE_general_shortnotation} 
from Theorem \ref{Lipschitzexistence}.
Then we have
$$\PP\left((u(t,x),v(t,x))\in [0,1]^2 \text{ for all } t\geq 0\text{  and }x \in \R \right) =1.$$
In particular, almost surely the solution has paths taking values in $C(]0,\infty[, C(\R,[0,1]))\times C([0,\infty[, B(\R,[0,1]))$, where the spaces are endowed with the topology of locally uniform convergence.

\end{thm}

\begin{remark}
By the same approximation procedure as in the proof of Theorem \ref{existencereal} below, it is possible to relax the condition in Theorem \ref{comparison} that $b, \tilde b$ resp.\ $\sigma$ are Lipschitz continuous jointly in $(x,u,v)$ resp.\ $(x,u)$ and to require merely joint continuity and the Lipschitz condition \eqref{eq:Lipschitzcondition}. 

\end{remark}

\noindent In order to extend Theorem \ref{comparison} 
to non-Lipschitz diffusion coefficients $\sigma$, we
need to impose an additional assumption on our SPDE in what follows. 
For given $u\in C(]0,\infty[\times\R,[0,1])$ and fixed $y \in \R$, we consider 
Equation \eqref{spde_v} as an ordinary 
integral equation in $v(\cdot, y)$. 
We then assume 
in effect that the unique solution $v$ is a deterministic functional of $u$ and $v_0$, in the sense of \eqref{eq: SPDE_general_shortnotation_Delay} below. We are then in a position to prove: 

\begin{thm} \label{existencereal}
Assume that $ b,\tilde b $ and $\sigma$ satisfy the following:
\begin{enumerate}
    \item[(i)] We have that $b, \tilde b \colon \R^2 \to \R$ are Lipschitz-continuous functions of $(u,v)$ alone
and satisfy the inequalities in \eqref{eq:inequality_b}.

\item[(ii)] We have that $\sigma \colon \R \to \R $ is a continuous (not necessarily Lipschitz) function of $u$ alone which 
satisfies the conditions in \eqref{conditionspositive} 
and a linear growth bound 
    \begin{align*}
    |\sigma(u)| \leq K (1+|u|)
\end{align*}
for all $u \in \R $ and some $K>0$.
\item[(iii)] 
We assume that there exist continuous 
 functionals 
\begin{align*}
    F &\colon C(]0,\infty[\times\R,[0,1]) \to C([0,\infty[ \times \R, [0,1]), 
    \\
    H &\colon [0,\infty[ \times  B(\R, [0,1])\to  B(\R,[0,1])
\end{align*}
such that for each given $u \in C(]0,\infty[\times\R,[0,1])$, $v_0 \in B(\R,[0,1])$ and $y\in\R$, the unique solution $v(\cdot,y)$ to Equation \eqref{spde_v} has the following representation:
\begin{align} \label{eq: SPDE_general_shortnotation_Delay}
     v(t,y) &= F(u)(t,y) + H(t,v_0)(y),\qquad t\ge0.
\end{align}
\end{enumerate}
%
Then for given initial conditions $u_0,v_0\in B(\R,[0,1])$ there exists a white noise process $W$ and a corresponding filtered probability space such that Equation \eqref{eq: SPDE_general_shortnotation} 
has a 
weak solution $(u,v)$ with paths in $C(]0,\infty[, C(\R,[0,1]))\times C([0,\infty[, B(\R,[0,1]))$
almost surely.
\end{thm}

\begin{remark}
Note that we have to impose the condition \eqref{eq: SPDE_general_shortnotation_Delay} that $v$ is a deterministic functional of $u$ (and $v_0)$ in order to reduce our coupled system of equations \eqref{eq: SPDE_general_shortnotation} to an equation in $u$ only. This is due to the fact that the methods we employ for tightness require Polish spaces and $B(\R,[0,1])$ (the state space of $v$) is not separable. 
\end{remark}

\noindent Our next goal is to show that our specific model, i.e. the SPDE \eqref{eq:SteppingStoneWithSeedbank} fits into the framework of the preceding theorems.
To this end, we first prove 
Proposition \ref{prop: Delay}, which allows us to represent our system of SPDEs as a single Stochastic Partial Delay Differential Equation (SPDDE).

\begin{proof} [Proof of Proposition \ref{prop: Delay}]
For given $u\in C(]0,\infty[\times\R,[0,1])$ and $v_0\in B(\R,[0,1])$, we consider for each fixed $x\in\R$ the second component of the system \eqref{eq:SteppingStoneWithSeedbank} as  an integral equation in $v(\cdot,x)$, 
i.e.
$$v(t,x)= v_0(x) + c'\int_0^t (u(s,x)-v(s,x)) \, \ddd s.$$
Then by an application of the variation of constants formula we get 
for all $x\in \R$ that
\begin{equation}\label{variation of constants}
    v(t,x)=e^{-c't} \left(c'  \int_0^t u(s,x) e^{c's} \, \ddd s +v_0(x) \right),\qquad t\ge0 .
\end{equation}
One may verify this through a simple application of the integration by parts formula. 
To see this we calculate as follows for each $x \in \R$:
\begin{align*}
    e^{c' t} v(t,x) &=v_0(x) + c' \int_0^t v(s,x) e^{c's} \, \ddd s+ \int_0^t e^{c's} \ddd v(s,x)\\
    &=v_0(x) + c' \int_0^t v(s,x) e^{c's} \, \ddd s + c'\int_0^t e^{c's} (u(s,x)-v(s,x)) \ddd s\\
    &= v_0(x)+c'\int_0^t e^{c's} u(s,x) \, \ddd s.
\end{align*}
 Rearranging we obtain \eqref{variation of constants}, which we note is just the integral form of the second equation in \eqref{eq:SteppingStoneWithSeedbankDelay}.
Now it is easy to see 
that $(u,v)\in C(]0,\infty[,C(\R,[0,1]))\times C([0,\infty[,B(\R,[0,1]))$ is a solution of \eqref{eq:SteppingStoneWithSeedbank} in the sense of Definition \ref{def: weaksolution} iff 
it is a solution of the SPDDE \eqref{eq:SteppingStoneWithSeedbankDelay}.
\end{proof}

\noindent We are finally in a position to provide the following:

\begin{proof} [Proof of Theorem \ref{thm: existence}]
Consider the SPDE given by
\begin{align*}
    \partial_t u(t,x) &= \frac{\Delta}{2} u(t,x) + s (u^2(t,x)-u(t,x))\ind_{[0,1]} (u(t,x)) +c (v(t,x)-u(t,x)) \\
    &\quad- m_1 u(t,x)+  m_2 (1-u(t,x)) \\
    &\quad + \ind_{[0,1]} (u(t,x))\sqrt{\nu u(t,x)(1-u(t,x))} \dot W (t,x), \\
    \partial_t v(t,x) &=  c'( u(t,x)-v(t,x)).
\end{align*}
Then we note that $$b(u,v):=(u^2-u) \ind_{[0,1]}(u) - m_1 u+  m_2 (1-u)+c(v \ind_{[0,1]}(v) +\ind_{]1,\infty[}(v)-u),
$$ 
$$\tilde b(u,v):= c'(u \ind_{[0,1]}(u) +\ind_{]1,\infty[}(u)-v)\qquad\text{and}\qquad \sigma(u) :=\ind_{[0,1]}(u)\sqrt{\nu u(1-u)}$$ 
satisfy all assumptions of Theorem \ref{existencereal}.
Moreover, \eqref{variation of constants} shows that \eqref{eq: SPDE_general_shortnotation_Delay} holds with 
\begin{align*}
F(u)(t,y)&:=e^{-c't} c' \int_0^t e^{c's} u(s,y) \, \ddd s,\\
H(t,v_0)(y) &:= e^{-c't} v_0(y).
\end{align*}
Thus 
all conditions of Theorem \ref{existencereal} are fulfilled and we have existence of a $[0,1]^2$-valued weak solution $(u,v)$ with paths in $ C(]0,\infty[,C(\R,[0,1]))\times C([0,\infty[,B(\R,[0,1]))$ almost surely. This means in turn that we may get rid of the indicator functions and hence $(u,v)$ solves Equation \eqref{eq:SteppingStoneWithSeedbank}.
\end{proof}

\begin{remark}
Note that the preceding results become only notationally harder to prove if we consider larger systems of equations of the following form:\\
Let $d,\tilde d,r \in \N$ and denote by
\begin{align*}
   & b\colon  [0,\infty [ \times \R \times \R^{d+\tilde d} \to \R^d,
   & \tilde b \colon  [0,\infty [ \times \R \times \R^{ d+\tilde d} \to \R^{\tilde d}
\end{align*}
and 
\begin{align*}
    \sigma \colon  [0,\infty [ \times \R \times \R^{d+\tilde d} \to \R^{d \times r}
\end{align*}
measurable maps. Then we may consider the system of SPDEs
\begin{align} \label{eq: SPDE_general_shortnotation1}
    \partial_t u(t,x) &= \frac{\Delta}{2} u(t,x)+b(t,x,u(t,x),v(t,x))  +\sigma(t,x,u(t,x),v(t,x)) \dot W(t,x), \nonumber\\
    \partial_t v(t,x) &= \tilde b(t,x,u(t,x),v(t,x))
\end{align}
where $W=(W_k)_{k=1, \ldots , r}$ is a collection of independent $1$-dimensional space-time white noise processes.\\
Written component-wise, \eqref{eq: SPDE_general_shortnotation1} means
\begin{align} \label{eq: SPDE_general_shortnotationcoordinate}
    \partial_t u_i(t,x) &= \frac{\Delta}{2} u_i(t,x)+b_i(t,x,u(t,x),v(t,x))  + \sum_{k=1}^r \sigma_{ik}(t,x,u(t,x), v(t,x)) \dot W_k(t,x), \nonumber\\
    \partial_t v_j(t,x) &= \tilde b_j(t,x,u(t,x),v(t,x))
\end{align}
for $i\in\lbrace 1, \ldots, d \rbrace$ and $j\in\lbrace 1, \ldots ,\tilde d \rbrace$.
\end{remark}


\section{Uniqueness in law and duality}
\label{sec:Uniqueness_and_duality}

\noindent In this section we aim to establish a moment duality which in particular will yield uniqueness in law for Equation \eqref{eq:SteppingStoneWithSeedbank}. We follow the approach of Athreya and Tribe in \cite{AT00}.
Now recall that on the one hand it is a well-known result by Shiga (cf.\ \cite{S88}) that the SPDE without seedbank given by
\begin{align}
    \partial_t u(t,x) &= \frac{\Delta}{2} u(t,x) + s  ((u^2(t,x)-u(t,x)) \notag \\
    &\quad- m_1 u(t,x)+ m_2 (1-u(t,x)) +  \sqrt{\nu u(t,x)(1-u(t,x)} \dot W (t,x)
\end{align}
satisfies a moment duality with a branching coalescing Brownian motion with killing $\tilde M$ and mutation compensator given by
 \begin{align*}
     \E\left[ \prod_{\beta \in \tilde I_0} u(t,\tilde M^\beta_0)   \right]&= \E\left[ \prod_{\beta \in \tilde I_t} u(0,\tilde M^\beta_t)   e^{-m_1 \int_0^t \abs{\tilde I_s} \, \ddd s}\right] 
 \end{align*}
where $\tilde I_t$ is the index set of particles that are alive at time $t\geq 0$.
On the other hand, in the seed bank diffusion case given by Equation \eqref{eq:seed_bank_diffusion} it has been established that the dual process is an ``on/off" version of a (Kingman-)coalescent (cf.\ \cite{BEGCKW15}).\\
In lieu of this we have defined a combination of all the previous mechanisms by allowing the movement of the particles of Shiga's dual process to also be subject to an additional ``on/off" mechanism as in the on/off coalescent case in Definition \ref{defn:dualprocess}. Recall that we denote by $M=(M_t)_{t \geq 0}$ an on/off BCBM with killing and that $I_t$ and $J_t$ are the index sets of active and dormant particles, respectively, at time $t \geq 0$.

\begin{remark}
In the discrete setting the proof of duality can usually be reduced to a simple generator  calculation using the arguments in \cite[pages 188-190]{EK86}. However, due to the involvement of the collision local time stemming from the coalescence mechanism in the dual process, the generator of the dual process can only be defined formally. Hence, we use a regularization procedure to still be able to use the main ideas from \cite{EK86}.
\end{remark}

\begin{proof}[Proof of Theorem \ref{thm:duality}]
Consider for $\eps >0$ the Gaussian heat kernel $\rho_\eps (x)= G(\eps, x,0)\break=\frac{1}{\sqrt{2 \pi \eps}}\exp\left(-\frac{x^2}{2\eps} \right)$ and set
\begin{align*}
    u_\eps (t,y) = \int_\R u (t,x) \rho_\eps(x-y) \, \ddd x 
\end{align*}
and note that $u_\eps (t , \cdot )$ is smooth for each $t \geq 0$. Then by Definition \ref{def: weaksolution} $u_\eps $ satisfies the integral equation 
\begin{align*}
    u_\eps (t,x)  &= u_\eps(0,x)+ \int_0^t \int_\R  \frac{1}{2}   u (s,y) \Delta \rho_\eps (x-y) \, \ddd y \, \ddd s+  m_2 \int_0^t\int_\R (1-u (s,y)) \rho_\eps (x-y) \, \ddd y \, \ddd s\\
    &\quad -m_1 \int_0^t \int_\R  u (s,y) \rho_\eps(x-y) \, \ddd y \, \ddd s -s  \int_0^t\int_\R (1-  u (s,y))  u (s,y) \rho_\eps (x-y) \, \ddd y \, \ddd s\\
    &\quad +c \int_0^t\int_\R ( v(s,y)-  u (s,y)) \rho_\eps(x-y) \, \ddd y \, \ddd s\\
    &\quad +  \sqrt{\nu}\int_0^t \int_\R \rho_\eps(x-y) \sqrt{  u(s,y) (1- u(s,y))} \, W (\ddd s,\ddd y) \nonumber\\
      &= u_\eps(0,x)+ \int_0^t \frac{1}{2} \Delta  u_\eps (s,x) +  m_2  (1-u_\eps (s,x)) -m_1   u_\eps (s,x) \, \ddd s  \\
      &\quad-s  \int_0^t\int_\R (1-  u (s,y))  u (s,y) \rho_\eps (x-y) \, \ddd y \, \ddd s\nonumber\\
      &\quad +c \int_0^t v_\eps (t,x)-  u_\eps (t,x) \, \ddd s
    +  \sqrt{\nu }\int_0^t \int_\R \rho_\eps(x-y) \sqrt{ u(s,y) (1- u(s,y))} \,  W (\ddd s,\ddd y) \nonumber \\
    &=u_\eps(0,x)+  \int_0^t \frac{1}{2} \Delta  u_\eps (s,x) +  m_2  (1-u_\eps (s,x)) -m_1   u_\eps (s,x) \, \ddd s  -s\int_0^t  b_\eps(u,s,x) \, \ddd s\nonumber\\
     &\quad +c \int_0^t ( v_\eps (s,x)-  u_\eps (s,x)) \, \ddd s  +  \sqrt{\nu}\int_0^t \int_\R \rho_\eps(x-y) \sqrt{  u(s,y) (1- u(s,y))} \, W (\ddd s,\ddd y) ,\nonumber \\
    v_\eps (t,x) &=v_\eps(0,x)+c' \int_0^t\int_\R ( u (s,y)-  v(s,y)) \rho_\eps(x-y) \, \ddd y \, \ddd s  \nonumber\\
    &= v_\eps(0,x)+  c' \int_0^t u_\eps (t,x)-  v_\eps (t,x) \, \ddd s
\end{align*}
where $b_\eps(u,s,x)\defeq  \int_\R (1-u(s,y))u(s,y) \rho_\eps (x-y) \, \ddd y $.
Note that the above two quantities $u_\eps$ and $v_\eps$ are semimartingales. Thus, taking $n,m \in \N$ and choosing arbitrary points $x_1, \ldots, x_n \in \R$ and $y_1 , \ldots , y_m \in \R$ we see by an application of  It\^{o}'s formula to the $C^2$ map
\[(x_1,\ldots,x_n,y_1,\ldots,y_m)\mapsto\prod_{i=1}^{n} x_i \prod_{j=1}^{m}y_j \]
that after taking expectations
\begin{align*}
	&\E\left[  \prod_{i=1}^n u_\eps (t,x_i) \prod_{j=1}^m v_\eps (t, y_j) \right] -\E\left[ \prod_{i=1}^n u_\eps (0,x_i) \prod_{j=1}^m v_\eps (0, y_j) \right]\nonumber\\
	&\quad = \E \left[ \int_0^t \sum_{i=1}^n \prod_{k=1, k \neq i}^n u_\eps (s,x_k) \prod_{j=1}^m v_\eps (s, y_j) \right.\nonumber\\
	&\quad \quad \quad \quad \quad \quad\qquad \qquad\qquad\qquad \times\left.\left( \frac{1}{2} \Delta  u_\eps (s,x_i) +  m_2  (1-u_\eps (s,x_i)) -m_1   u_\eps (s,x_i)  \right) \, \ddd s \right]\nonumber\\
	&\quad \quad+ c\E \left[  \int_0^t \sum_{i=1}^n \prod_{k=1, k \neq i}^n u_\eps (s,x_k) \prod_{j=1}^m v_\eps (s, y_j) \left(  v_\eps (s,x_i)-  u_\eps (s,x_i)  \right) \, \ddd s \right]\nonumber\\
	&\quad \quad- s\E \left[  \int_0^t \sum_{i=1}^n \prod_{k=1, k \neq i}^n u_\eps (s,x_k) \prod_{j=1}^m v_\eps (s, y_j)     b_\eps(u,s,x_i)  \, \ddd s \right]\nonumber\\
	&\quad \quad  + c'\E \left[ \int_0^t \sum_{j=1}^m \prod_{i=1}^n u_\eps (s,x_i) \prod_{k=1,k \neq j }^m v_\eps (s, y_k) \left(   u_\eps (s,y_j)-  v_\eps (s,y_j)  \right) \, \ddd s \right]\nonumber\\
	&\quad \quad +\frac{1}{2}\nu \E \left[\int_0^t \sum_{i=1,l=1, i\neq l}^n \prod_{k\in \lbrace 1, \ldots , n \rbrace \setminus \lbrace i,l \rbrace}^m u_\eps (s,x_k) \prod_{j=1}^m v_\eps (s,y_j) \right.\nonumber\\
	&\quad \quad \quad \quad \quad \quad \qquad\qquad\qquad\qquad \times\left.\int_\R \rho_\eps(z-x_i) \rho_\eps(z-x_l)  \sigma^2(u(s,z))\, \ddd z \, \ddd s\right]
\end{align*}
where $\sigma(x)=\sqrt{ x(1-x)}$.
Now, we replace the $x_i$ and $y_j$ by an independent version of our dual process taken at a time $r\geq 0$ and multiply by the independent quantity $K(r):=e^{-m_1 \int_0^r |I_s| \, \ddd s}$. This gives 
\begin{align} \label{eq:ItoforSPDE}
	&\E\left[K(r) \prod_{\beta \in I_r} u_\eps (t,M^\beta_r) \prod_{\gamma \in J_r} v_\eps (t, M^\gamma_r) \right] -\E\left[ K(r)\prod_{\beta \in I_r} u_\eps (0,M^\beta_r) \prod_{\gamma \in J_r} v_\eps (0, M^\gamma_r) \right]\nonumber\\
	&\quad = \E \left[ K(r) \int_0^t \sum_{\beta \in I_r} \prod_{\delta \in I_r \setminus \lbrace \beta \rbrace} u_\eps (s,M^\delta_r) \prod_{\gamma \in J_r} v_\eps (s, M^\gamma_r) \right.\nonumber\\
	& \qquad\qquad\qquad\qquad\qquad\qquad\qquad\times\left.\left( \frac{1}{2} \Delta  u_\eps (s,M^\beta_r) +  m_2  (1-u_\eps (s,M^\beta_r)) -m_1   u_\eps (s,M^\beta_r)  \right) \, \ddd s \right]\nonumber\\
	&\qquad+ c\E \left[ K(r) \int_0^t \sum_{\beta \in I_r} \prod_{\delta \in I_r \setminus \lbrace \beta \rbrace} u_\eps (s,M^\delta_r) \prod_{\gamma \in J_r} v_\eps (s, M^\gamma_r) \right.\nonumber \\
	&\qquad\qquad\qquad\qquad\qquad\qquad\qquad \times\left. \left(   v_\eps (s,M^\beta_r)-  u_\eps (s,M^\beta_r)\right) \, \ddd s \right]\nonumber\\
	&\qquad- s\E \left[ K(r) \int_0^t \sum_{\beta \in I_r} \prod_{\delta \in I_r \setminus \lbrace \beta \rbrace} u_\eps (s,M^\delta_r) \prod_{\gamma \in J_r} v_\eps (s, M^\gamma_r) \right.\nonumber \\
	&\qquad\qquad\qquad\qquad\qquad\qquad\qquad \times\left.    b_\eps(u,s,M^\beta_r)  \, \ddd s \right]\nonumber\\
	&\qquad  +c' \E \left[ K(r) \int_0^t \sum_{\gamma \in J_r}\prod_{\beta \in I_r} u_\eps (s,M^\beta_r) \prod_{\delta \in J_r \setminus \lbrace \gamma \rbrace} v_\eps (s, M^\delta_r) \left(    u_\eps (s,M^\gamma_r)) -v_\eps (s,M^\gamma_r) \right) \, \ddd s \right]\nonumber\\
	&\qquad + \frac{1}{2}\nu \E \left[K(r)\int_0^t \sum_{\beta, \delta \in I_r, \beta \neq \delta} \prod_{\phi \in I_r \setminus \lbrace\beta, \delta \rbrace} u_\eps (s,M^\phi_r) \prod_{\gamma \in J_r} v_\eps (s,M^\gamma_r)\right.\nonumber\\
	&\qquad\qquad\qquad\qquad\qquad\qquad\qquad\times \left.\int_\R \rho_\eps(z-M^\beta_r) \rho_\eps(z-M^\delta_r) \sigma^2(u(s,z)) \, \ddd z \, \ddd s\right].
\end{align}
Further, since we have the following integrable upper bound
\begin{align*}
&K(r)\int_0^t \sum_{\beta, \delta \in I_r, \beta \neq \delta} \prod_{\phi \in I_r \setminus \lbrace\beta, \delta \rbrace} u_\eps (s,M^\phi_r) \prod_{\gamma \in J_r} v_\eps (s,M^\gamma_r) \int_\R \rho_\eps(z-M^\beta_r) \rho_\eps(z-M^\delta_r)\sigma^2(u(s,z)) \, \ddd z \, \ddd s \\
&\quad \leq C(t,\eps) I_r
\end{align*}
we may use Fubini's theorem and similar bounds (which are possibly independent of $\eps$) for the remaining quantities to justify that the terms in \eqref{eq:ItoforSPDE} are finite. Note that this also allows for applications of the dominated convergence theorem later on.\\
On the other hand, for any $C^2$-functions $h,g$ we see by adding and substracting the compensators of the jumps 

\begin{align*}
&\E\left[ \prod_{\beta \in I_t}h(M^\beta_t) \prod_{\gamma \in J_t} g(M^\gamma_t) \right] -\E\left[ \prod_{i=1}^n h(x_i) \prod_{j=1}^m g( y_i) \right]\nonumber\\
&\quad= \E\left[ \int_0^t \sum_{\beta \in I_s}  \prod_{\delta \in I_s, \delta \neq \beta}h(M^\delta_s) \prod_{\gamma \in J_s}g(M^\gamma_s) \frac{1}{2} \Delta h(M^\beta_s) \, \ddd s\right]\nonumber\\
&\quad \quad +s\E\left[ \int_0^t \sum_{\beta \in I_s}  \prod_{\delta \in I_s, \delta \neq \beta}h(M^\delta_s) \prod_{\gamma \in J_s}g(M^\gamma_s) (h^2(M^\beta_s)-h(M^\beta_s))  \, \ddd s \right]\nonumber\\
&\quad \quad +c\E\left[ \int_0^t \sum_{\beta \in I_s}  \prod_{\delta \in I_s, \delta \neq \beta}h(M^\delta_s) \prod_{\gamma \in J_s}g(M^\gamma_s) (g(M_s^\beta)-h(M^\beta_s))  \, \ddd s \right]\nonumber\\
&\quad \quad +c'\E\left[ \int_0^t \sum_{\gamma \in J_s}  \prod_{\beta \in I_s} h(M^\beta_s) \prod_{\delta \in J_s, \delta \neq \gamma}g(M^\delta_s) (h(M_s^\gamma)-g(M^\gamma_s))  \, \ddd s \right]\nonumber\\
&\quad \quad + m_2\E\left[ \int_0^t \sum_{\beta \in I_s}  \prod_{\delta \in I_s, \delta \neq \beta}h(M^\delta_s) \prod_{\gamma \in J_s}g(M^\gamma_s) (1-h(M^\beta_s))  \, \ddd s \right]\nonumber\\
&\quad \quad +\frac{1}{4}\nu\E\left[ \int_0^t \sum_{\beta, \delta \in I_s, \beta \neq \delta}  \prod_{\phi \in I_s\setminus \lbrace \beta, \delta \rbrace}h(M^\phi_s) \prod_{\gamma \in J_s}g(M^\gamma_s)  (h(M^\delta_s)-h(M^\beta_s)h(M^\delta_s))  \, \ddd L^{\beta,\delta}_s \right]
\end{align*}
where we recall from Definition \ref{defn:dualprocess} that $L^{\beta,\delta}$ is the local time of $M^\beta-M^\delta$ at $0$ whenever both particles are active. 
By the integration by parts formula we then see including the factor $K(t)$
\begin{align*}
&\E\left[K(t) \prod_{\beta \in I_t}h(M^\beta_t) \prod_{\gamma \in J_t} g(M^\gamma_t) \right] -\E\left[ K(0) \prod_{i=1}^n h(x_i) \prod_{j=1}^m g( y_i) \right]\\
&\quad= \E\left[ \int_0^t K(s)\sum_{\beta \in I_s}  \prod_{\delta \in I_s, \delta \neq \beta}h(M^\delta_s) \prod_{\gamma \in J_s}g(M^\gamma_s) \frac{1}{2} \Delta h(M^\beta_s) \, \ddd s\right]\\
&\quad \quad +s\E\left[ \int_0^t K(s) \sum_{\beta \in I_s}  \prod_{\delta \in I_s, \delta \neq \beta}h(M^\delta_s) \prod_{\gamma \in J_s}g(M^\gamma_s) (h^2(M^\beta_s)-h(M^\beta_s))  \, \ddd s \right]\\
&\quad \quad +c\E\left[ \int_0^t K(s)\sum_{\beta \in I_s}  \prod_{\delta \in I_s, \delta \neq \beta}h(M^\delta_s) \prod_{\gamma \in J_s}g(M^\gamma_s) (g(M^\beta_s)-h(M^\beta_s))  \, \ddd s \right]\\
&\quad \quad +c'\E\left[ \int_0^t K(s) \sum_{\gamma \in J_s}  \prod_{\beta \in I_s} h(M^\beta_s) \prod_{\delta \in J_s, \delta \neq \gamma}g(M^\delta_s) (h(M^\gamma_s)-g(M^\gamma_s))  \, \ddd s \right]\\
&\quad \quad +m_2 \E\left[ \int_0^t K(s) \sum_{\beta \in I_s}  \prod_{\delta \in I_s, \delta \neq \beta}h(M^\delta_s) \prod_{\gamma \in J_s}g(M^\gamma_s) (1-h(M^\beta_s))  \, \ddd s \right]\\
&\quad \quad +\frac{1}{4}\nu\E\left[ \int_0^t K(s) \sum_{\beta, \delta \in I_s, \beta \neq \delta}  \prod_{\phi \in I_s\setminus \lbrace \beta, \delta \rbrace}h(M^\phi_s) \prod_{\gamma \in J_s}g(M^\gamma_s)  (h(M^\delta_s)-h(M^\beta_s)h(M^\delta_s))  \, \ddd L^{\beta,\delta}_s \right]\\
&\quad \quad + \E\left[ \int_0^t K(s) \prod_{\beta \in I_s}h(M^\beta_s) \prod_{\gamma \in J_s} g(M^\gamma_s) (-m_1 \abs{I_s}) \, \ddd s\right].
\end{align*}
Further, note that that for each $\eps >0$ and $r\geq 0$ the maps $u_\eps (r,\cdot),\, v_\eps (r, \cdot)$  are bounded and smooth as they originate from mollifying with the heat kernel. Thus, replacing $h$ and $g$ with the independent quantities $u_\eps (r,\cdot)$ and $v_\eps (r,\cdot)$ we have
\begin{align*}
&\E\left[K(t) \prod_{\beta \in I_t}u_\eps (r,M^\beta_t) \prod_{\gamma \in J_t} v_\eps (r,M^\gamma_t) \right] -\E\left[ K(0) \prod_{i=1}^n u_\eps (r,x_i) \prod_{j=1}^m v_\eps (r, y_j) \right]\\
&\quad= \E\left[ \int_0^t K(s)\sum_{\beta \in I_s}  \prod_{\delta \in I_s, \delta \neq \beta}u_\eps (r,M^\delta_s) \prod_{\gamma \in J_s}v_\eps (r,M^\gamma_s) \frac{1}{2} \Delta u_\eps (r,M^\beta_s) \, \ddd s\right]\\
&\quad \quad +s\E\left[ \int_0^t K(s) \sum_{\beta \in I_s}  \prod_{\delta \in I_s, \delta \neq \beta}u_\eps (r,M^\delta_s) \prod_{\gamma \in J_s}v_\eps (r,M^\gamma_s) (u_\eps ^2(r,M^\beta_s)-u_\eps (r,M^\beta_s))  \, \ddd s \right]\\
&\quad \quad +c\E\left[ \int_0^t K(s)\sum_{\beta \in I_s}  \prod_{\delta \in I_s, \delta \neq \beta}u_\eps (r,M^\delta_s) \prod_{\gamma \in J_s}v_\eps (r,M^\gamma_s) (v_\eps (r,M^\beta_s)-u_\eps (r,M^\beta_s))  \, \ddd s \right]\\
&\quad \quad +c'\E\left[ \int_0^t K(s) \sum_{\gamma \in J_s}  \prod_{\beta \in I_s} u_\eps (r,M^\beta_s) \prod_{\delta \in J_s, \delta \neq \gamma}v_\eps (r,M^\delta_s) (u_\eps (r,M^\gamma_s)-v_\eps (r,M^\gamma_s))  \, \ddd s \right]\\
&\quad \quad + m_2\E\left[ \int_0^t K(s) \sum_{\beta \in I_s}  \prod_{\delta \in I_s, \delta \neq \beta}u_\eps (r,M^\delta_s) \prod_{\gamma \in J_s}v_\eps (r,M^\gamma_s) (1-u_\eps (r,M^\beta_s))  \, \ddd s \right]\\
&\quad \quad +\frac{1}{4}\nu \E\left[ \int_0^t K(s) \sum_{\beta, \delta \in I_s, \beta \neq \delta}  \prod_{\phi \in I_s\setminus \lbrace \beta, \delta \rbrace}u_\eps (r,M^\phi_s) \prod_{\gamma \in J_s}v_\eps (r,M^\gamma_s)\right.\\
&\qquad\qquad\qquad\qquad\qquad\qquad\qquad\times \left. (u_\eps (r,M^\delta_s)-u_\eps (r,M^\beta_s)u_\eps (r,M^\delta_s))  \, \ddd L^{\beta,\delta}_s \right]\\
&\quad \quad + \E\left[ \int_0^t K(s) \prod_{\beta \in I_s}u_\eps (r,M^\beta_s) \prod_{\gamma \in J_s} v_\eps (r,M^\gamma_s) (-m_1 \vert I_s\vert ) \, \ddd s\right].
\end{align*}
Finally, define
\begin{align*}
k(t,s,\eps ) =  \E\left[\prod_{\beta \in I_s} u_\eps (t,M_s^\beta) \prod_{\gamma \in J_s} v_\eps (t,M^\gamma_s)\right].
\end{align*}
Then we may follow the idea from \cite{EK86} and calculate for $t\geq 0$ by substituting with our previously calculated quantities
\begin{align}
&\int_0^t k(r,0,\eps )-k(0,r, \eps) \, \ddd r \nonumber\\
&\quad = \int_0^t  k(t-r,r,\eps)-k(0,r,\eps) \, \ddd r - \int_0^tk(r,t-r,\eps )- k(r,0,\eps ) \, \ddd r \nonumber\\
&\quad = \E \left[  \int_0^t \int_0^{t-r}K(r) \sum_{\beta \in I_r} \prod_{\delta \in I_r \setminus \lbrace \beta \rbrace} u_\eps (s,M^\delta_r) \prod_{\gamma \in J_r} v_\eps (s, M^\gamma_r)  \frac{1}{2} \Delta  u_\eps (s,M^\beta_r)   \, \ddd s\, \ddd r \right]\nonumber\\
&\quad\quad -\E\left[ \int_0^t \int_0^{t-r} K(s)\sum_{\beta \in I_s}  \prod_{\delta \in I_s, \delta \neq \beta}u_\eps (r,M^\delta_s) \prod_{\gamma \in J_s}v_\eps (r,M^\gamma_s) \frac{1}{2} \Delta u_\eps (r,M^\beta_s) \, \ddd s \, \ddd r\right]\nonumber\\
&\quad - s\E \left[   \int_0^t \int_0^{t-r}K(r) \int_0^{t-r}\sum_{\beta \in I_r} \prod_{\delta \in I_r \setminus \lbrace \beta \rbrace} u_\eps (s,M^\delta_r) \prod_{\gamma \in J_r} v_\eps (s, M^\gamma_r)    b_\eps(u,s,M_r^\beta)   \, \ddd s \, \ddd r \right]\nonumber\\
&\quad \quad   -s\E\left[ \int_0^t \int_0^{t-r} K(s) \sum_{\beta \in I_s}  \prod_{\delta \in I_s, \delta \neq \beta}u_\eps (r,M^\delta_s) \prod_{\gamma \in J_s}v_\eps (r,M^\gamma_s) (u_\eps ^2(r,M^\beta_s)-u_\eps (r,M^\beta_s))  \, \ddd s \, \ddd r \right]\nonumber\\
&\quad+ c\E \left[  \int_0^t \int_0^{t-r}K(r)  \sum_{\beta \in I_r} \prod_{\delta \in I_r \setminus \lbrace \beta \rbrace} u_\eps (s,M^\delta_r) \prod_{\gamma \in J_r} v_\eps (s, M^\gamma_r) \left(   v_\eps (s,M^\beta_r)-  u_\eps (s,M^\beta_r)  \right) \, \ddd s \, \ddd r \right]\nonumber\\
&\quad \quad  -c\E\left[ \int_0^t\int_0^{t-r} K(s)\sum_{\beta \in I_s}  \prod_{\delta \in I_s, \delta \neq \beta}u_\eps (r,M^\delta_s) \prod_{\gamma \in J_s}v_\eps (r,M^\gamma_s) (v_\eps (r,M^\beta_s)-u_\eps (r,M^\beta_s))  \, \ddd s  \, \ddd r\right]\nonumber\\
&\quad  + c'\E \left[  \int_0^t\int_0^{t-r}K(r)  \sum_{\gamma \in J_r}\prod_{\beta \in I_r} u_\eps (s,M^\beta_r) \prod_{\delta \in J_r \setminus \lbrace \gamma \rbrace} v_\eps (s, M^\delta_r) \left(     u_\eps (s,M^\gamma_r)-v_\eps (s,M^\gamma_r)  \right) \, \ddd s \, \ddd r \right]\nonumber\\
&\quad \quad -c'\E\left[ \int_0^t \int_0^{t-r} K(s) \sum_{\gamma \in J_s}  \prod_{\beta \in I_s} u_\eps (r,M^\beta_s) \prod_{\delta \in J_s, \delta \neq \gamma}v_\eps (r,M^\delta_s) (u_\eps (r,M^\gamma_s)-v_\eps (r,M^\gamma_s))  \, \ddd s \, \ddd r \right]\nonumber\\
&\quad + m_2\E \left[  \int_0^t \int_0^{t-r} K(r) \sum_{\beta \in I_r} \prod_{\delta \in I_r \setminus \lbrace \beta \rbrace} u_\eps (s,M^\delta_r) \prod_{\gamma \in J_r} v_\eps (s, M^\gamma_r)     (1-u_\eps (s,M^\beta_r))   \, \ddd s \, \ddd r\right]\nonumber\\
&\quad \quad -m_2\E\left[ \int_0^t \int_0^{t-r} K(s) \sum_{\beta \in I_s}  \prod_{\delta \in I_s, \delta \neq \beta}u_\eps (r,M^\delta_s) \prod_{\gamma \in J_s}v_\eps (r,M^\gamma_s) (1-u_\eps (r,M^\beta_s))  \, \ddd s\, \ddd r \right]\label{eq:death}\\
&\quad + \frac{1}{2}\nu\E \left[\int_0^t \int_0^{t-r}K(r)  \sum_{\beta, \delta \in I_r, \beta \neq \delta} \prod_{\phi \in I_r \setminus \lbrace\beta, \delta \rbrace} u_\eps (s,M^\phi_r) \prod_{\gamma \in J_r} v_\eps (s,M^\gamma_r)\right.\nonumber\\
&\qquad\qquad\qquad\qquad\qquad\qquad\times\left. \int_\R \rho_\eps(z-M^\beta_r) \rho_\eps(z-M^\delta_r)\sigma^2(u(s,z)) \, \ddd z \, \ddd s \, \ddd r\right] \nonumber\\
&\quad \quad  -\frac{1}{4}\nu\E\left[ \int_0^t \int_0^{t-r} K(s) \sum_{\beta, \delta \in I_s, \beta \neq \delta}  \prod_{\phi \in I_s\setminus \lbrace \beta, \delta \rbrace}u_\eps (r,M^\phi_s) \prod_{\gamma \in J_s}v_\eps (r,M^\gamma_s) \right.\nonumber\\
&\qquad\qquad\qquad\qquad\qquad\qquad\times\left.  (u_\eps (r,M^\delta_s)-u_\eps (r,M^\beta_s)u_\eps (r,M^\delta_s))  \, \ddd L^{\beta,\delta}_s \, \ddd r \right]\nonumber\\
&\quad +\E \left[   \int_0^t \int_0^{t-r}K(r) \sum_{\beta \in I_r} \prod_{\delta \in I_r \setminus \lbrace \beta \rbrace} u_\eps (s,M^\delta_r) \prod_{\gamma \in J_r} v_\eps (s, M^\gamma_r) \left(  -m_1   u_\eps (s,M^\beta_r)  \right)  \, \ddd s \, \ddd r\right]\nonumber\\
&\quad \quad  - \E\left[ \int_0^t\int_0^{t-r}  K(s) \prod_{\beta \in I_s}u_\eps (r,M^\beta_s) \prod_{\gamma \in J_s} v_\eps (r,M^\gamma_s) (-m_1 \vert I_s\vert) \, \ddd s \, \ddd r\right]\nonumber\\
&= -s\E \left[   \int_0^t \int_0^{t-r}K(r) \int_0^{t-r}\sum_{\beta \in I_r} \prod_{\delta \in I_r \setminus \lbrace \beta \rbrace} u_\eps (s,M^\delta_r) \prod_{\gamma \in J_r} v_\eps (s, M^\gamma_r)    b_\eps(u,s,M_r^\beta)     \, \ddd s \, \ddd r \right]\\
&\quad \quad   -s\E\left[ \int_0^t \int_0^{t-r} K(s) \sum_{\beta \in I_s}  \prod_{\delta \in I_s, \delta \neq \beta}u_\eps (r,M^\delta_s) \prod_{\gamma \in J_s}v_\eps (r,M^\gamma_s) (u_\eps ^2(r,M^\beta_s)-u_\eps (r,M^\beta_s))  \, \ddd s \, \ddd r \right]\label{problem0}\\
&\quad + \frac{1}{2}\nu\E \left[\int_0^t \int_0^{t-r}K(r)  \sum_{\beta, \delta \in I_r, \beta \neq \delta} \prod_{\phi \in I_r \setminus \lbrace\beta, \delta \rbrace} u_\eps (s,M^\phi_r) \prod_{\gamma \in J_r} v_\eps (s,M^\gamma_r) \right. \nonumber \\
&\qquad\qquad\qquad\qquad\qquad\qquad\qquad\times\left. \int_\R \rho_\eps(z-M^\beta_r) \rho_\eps(z-M^\delta_r) \sigma^2(u(s,z)) \, \ddd z \, \ddd s \, \ddd r\right] \label{problem1}\\
&\quad \quad  -\frac{1}{4}\nu\E\left[ \int_0^t \int_0^{t-r} K(s) \sum_{\beta, \delta \in I_s, \beta \neq \delta}  \prod_{\phi \in I_s\setminus \lbrace \beta, \delta \rbrace}u_\eps (r,M^\phi_s) \prod_{\gamma \in J_s}v_\eps (r,M^\gamma_s)\right.\nonumber\\
&\qquad\qquad\qquad\qquad\qquad\qquad\qquad\times\left.  (u_\eps (r,M^\delta_s)-u_\eps (r,M^\beta_s)u_\eps (r,M^\delta_s))  \, \ddd L^{\beta,\delta}_s \, \ddd r \right].\label{problem2}
\end{align}
Now, again by Fubini's theorem, taking the limit as $\eps \to 0$ we see that all the quantities except \eqref{problem1} and \eqref{problem2} vanish. Note that we used properties of mollifications (since $u$ is continuous we have convergence everywhere) and the dominated convergence theorem to justify this. Moreover, the use of Lemma \ref{lemma1} allows us to argue that the same is true for the remaining two terms. By the (left-)continuity of the the processes $u,v,M$ in $r$ we finally see that after differentiation
\begin{align*}
k(t,0,0)=k(0,t,0)
\end{align*}
as desired.
\end{proof}

\begin{lemma} \label{lemma1}
In the setting of Theorem \ref{thm:duality} we have
\begin{align*}
&\frac{1}{2}\nu\E \left[\int_0^t \int_0^{t-r}K(r)  \sum_{\beta, \delta \in I_r, \beta \neq \delta} \prod_{\phi \in I_r \setminus \lbrace\beta, \delta \rbrace} u_\eps (s,M^\phi_r) \prod_{\gamma \in J_r} v_\eps (s,M^\gamma_r)\right.\nonumber\\
&\qquad\qquad\qquad\qquad\qquad\qquad\qquad\times\left. \int_\R \rho_\eps(z-M^\beta_r) \rho_\eps(z-M^\delta_r)\sigma^2(u(s,z)) \, \ddd z \, \ddd s \, \ddd r\right] \\
&\quad \quad  \to \frac{1}{4}\nu\E\left[ \int_0^t \int_0^{t-r} K(s) \sum_{\beta, \delta \in I_s, \beta \neq \delta}  \prod_{\phi \in I_s\setminus \lbrace \beta, \delta \rbrace}u (r,M^\phi_s) \prod_{\gamma \in J_s}v (r,M^\gamma_s) \right. \nonumber\\
&\qquad\qquad\qquad\qquad\qquad\qquad\times \left.  (u (r,M^\delta_s)-u (r,M^\beta_s)u (r,M^\delta_s))  \, \ddd L^{\beta,\delta}_s \, \ddd r \right]
\end{align*}
as $\eps \to 0$. 
\end{lemma}

\begin{proof}[Proof of Lemma \ref{lemma1}]
We adapt and elaborate the proof of \cite{AT00}. Consider for each $m \in \N$ the time of the $m$-th birth $\tau_m$.
Assume for now that we may restrict ourselves to $[0, \tau_m]$. Now, we argue pathwise in $\omega$. Set for $\beta, \delta \in \lbrace 1, \ldots, m \rbrace$ after a change of variables
\begin{align*}
&\frac{1}{2}\int_0^t \int_0^{t-r} \ind_{\lbrace \beta, \delta \in I_r \rbrace}K(r) \prod_{\phi \in I_r \setminus \lbrace\beta, \delta \rbrace} u_\eps (s,M^\phi_r) \prod_{\gamma \in J_r} v_\eps (s,M^\gamma_r) \nonumber\\
&\qquad\qquad\qquad\qquad\qquad\qquad\qquad\times \int_\R \rho_\eps(z-M^\beta_r) \rho_\eps(z-M^\delta_r)\sigma^2(u(s,z)) \, \ddd z \, \ddd s \, \ddd r\\
&=\colon \frac{1}{2}  \int_0^t \int_\R \rho_\eps (z) \rho_\eps(z+ M^\delta_r-M^\beta_r) Z_r^{\beta,\delta, \eps}(z) \, \ddd z \, \ddd r.
\end{align*}
Here, we set 
\begin{align*}
    Z_r^{\beta,\delta, \eps}(z) &\defeq\int_0^{t-r} \ind_{\lbrace \beta, \delta \in I_r \rbrace}  K(r) \prod_{\phi \in I_r \setminus \lbrace\beta, \delta \rbrace} u_\eps (s,M^\phi_r) \prod_{\gamma \in J_r} v_\eps (s,M^\gamma_r) \sigma^2(u(s,z+M_r^\delta)) \, \ddd s, \\
    Z_r^{\beta,\delta}(z) &\defeq \int_0^{t-r} \ind_{\lbrace \beta, \delta \in I_r \rbrace}  K(r) \prod_{\phi \in I_r \setminus \lbrace\beta, \delta \rbrace} u (s,M^\phi_r) \prod_{\gamma \in J_r} v (s,M^\gamma_r) \sigma^2(u(s,z+M_r^\delta)) \, \ddd s.
\end{align*}
Now, by a modification of Tanaka's occupation time formula (cf.\ \cite{AT00}) we see that\footnote{Here we split the integral into the random time intervals on which both $M^\delta$ and $M^\beta$ are active and then apply Tanaka's formula. Moreover, note that the intersection local time of two on/off Brownian motions can always be dominated by the local time of two corresponding standard Brownian motions. This will also allow for applications of the local time inequalities by Barlow and Yor later in the proof.}
\begin{align*}
\frac{1}{2}\int_0^t \int_\R \rho_\eps (z) \rho_\eps(z+ M^\delta_r-M^\beta_r)  Z_r^{\beta,\delta, \eps}(z) \, \ddd z \, \ddd r = \frac{1}{4}  \int_\R \int_\R\int_0^t \rho_\eps (z) \rho_\eps(z+x)  Z_r^{\beta,\delta, \eps}(z)\, \ddd L^{\beta, \delta}_{r,x} \, \ddd x \, \ddd z.
\end{align*}
The goal for now is to replace $L^{\beta, \delta}_{r,x}$ by the local time at $0$ and get rid of the dependence of $Z$ on $z$ and $\eps$. To do so we first use the triangle inequality to see that 
\begin{align}
&\left\vert \int_\R \int_\R \int_0^t\rho_\eps (z) \rho_\eps(z+x) Z_r^{\beta,\delta, \eps}(z) \, \ddd L^{\beta, \delta}_{r,x} \, \ddd x \, \ddd z -\int_0^t  Z_r^{\beta, \delta} (0) \, \ddd L^{\beta, \delta}_r \right\vert \nonumber\\
&\leq \int_\R \int_\R\int_0^t \rho_\eps (z) \rho_\eps(z+x) \left|Z_r^{\beta,\delta, \eps}(z)-Z_r^{\beta,\delta , \eps}(0) \right|\, \ddd L^{\beta, \delta}_{r,x} \, \ddd x \, \ddd z\label{eq:firstterm}\\
&\qquad + \int_\R \int_\R\int_0^t \rho_\eps (z) \rho_\eps(z+x) \left|Z_r^{\beta,\delta, \eps}(0)-Z_r^{\beta,\delta}(0) \right|\, \ddd L^{\beta, \delta}_{r,x} \, \ddd x \, \ddd z\label{eq:secondterm}\\
&\qquad +\left| \int_\R \int_\R \int_0^t \rho_\eps (z) \rho_\eps(z+x) Z_r^{\beta,\delta}(0)\, \ddd L^{\beta, \delta}_{r,x} \, \ddd x \, \ddd z- \int_0^t  Z_r^{\beta,\delta}(0)\, \ddd L^{\beta, \delta}_{r} \right| \label{eq:thirdterm}.
\end{align}
First, note that $r \mapsto Z^{\beta,\delta,\eps}_r(z)$ is piece-wise continuous and uniformly bounded for each $z \in \R$ and that $z \mapsto Z^{\beta,\delta}_r(z)$ is continuous at $0$ uniformly in $r \leq t$ and $\eps>0$. Hence, for any $\eta>0$ there exists some $\zeta >0$ such that
$$\left|Z_r^{\beta,\delta, \eps}(z)-Z_r^{\beta,\delta, \eps}(0) \right|<\eta$$
whenever $|z|<\zeta$ uniformly in $r$ and $\eps$. Using that $L_{r,x}$ is uniformly bounded in $r\in [0,t],x\in \R$ by Barlow and Yor's local time inequalities from \cite{B81} we obtain for all $\eps>0$
$$ \int_\R \int_\R \int_0^t\ind_{\lbrace \vert z\vert <\zeta \rbrace}(z) \rho_\eps (z) \rho_\eps(z+x) \left|Z_r^{\beta,\delta, \eps}(z)-Z_r^{\beta,\delta, \eps}(0) \right|\, \ddd L^{\beta, \delta}_{r,x} \, \ddd x \, \ddd z< C\eta$$
for some constant $C>0$. Moreover, noting in addition that $Z^{\beta, \delta, \eps}$ is bounded uniformly in $r, \eps$ and $z$ we have for the remaining part of the integral
\begin{align*}
    & \int_\R \int_\R\int_0^t \ind_{\lbrace \vert z\vert >\zeta \rbrace}(z) \rho_\eps (z) \rho_\eps(z+x) \left|Z_r^{\beta,\delta, \eps}(z)-Z_r^{\beta,\delta, \eps}(0) \right|\, \ddd L^{\beta, \delta}_{r,x} \, \ddd x \, \ddd z \\
    &\qquad \leq C \int_\R \ind_{\lbrace \vert z\vert >\zeta \rbrace}(z)  \rho_\eps(z) \, \ddd z \to 0
\end{align*}
as $\eps \to 0$. Hence, we obtain that the entire term in Equation \eqref{eq:firstterm} vanishes as $\eps \to 0$.

For the term \eqref{eq:secondterm} we set for $\varphi \in \lbrace 1,\ldots , m\rbrace $ and $\varphi \neq \beta, \delta$
\begin{align*}
    Z_r^{\beta,\delta, \eps}(0) &=\int_0^{t-r} \ind_{\lbrace \beta, \delta \in I_r \rbrace}  K(r) \prod_{\phi \in I_r \setminus \lbrace\beta, \delta \rbrace} u_\eps (s,M^\phi_r) \prod_{\gamma \in J_r} v_\eps (s,M^\gamma_r) \sigma^2(u(s,M_r^\delta)) \, \ddd s\\
    &=  \int_0^{t-r}\ind_{\lbrace \beta, \delta, \phi \in I_r \rbrace}  K(r) \int_\R u(s,M_r^{\varphi }-y) \rho_\eps(y)\, \ddd y \\
    &\qquad \qquad \times\prod_{\phi \in I_r \setminus \lbrace\beta, \delta, \varphi \rbrace} u_\eps (s,M^\phi_r) \prod_{\gamma \in J_r} v_\eps (s,M^\gamma_r) \sigma^2(u(s,M_r^\delta)) \, \ddd s \\
    &= \colon  \int_\R \rho_\eps(y) \bar Z_r^{\beta, \delta, \varphi}(y)\, \ddd y .
\end{align*}
By the same argument as for the term \eqref{eq:firstterm} we see that as $\eps \to 0$ that
\begin{align*}
    &\int_\R \int_\R\int_0^t \rho_\eps (z) \rho_\eps(z+x) \left|Z_r^{\beta,\delta, \eps}(0)-\int_\R \rho_\eps(y) \bar Z_r^{\beta, \delta, \varphi}(0)\, \ddd y \right|\, \ddd L^{\beta, \delta}_{r,x} \, \ddd x \, \ddd z\\
    &\qquad \leq \int_\R \int_\R\int_0^t \rho_\eps (z) \rho_\eps(z+x) \int_\R \rho_\eps(y) \left|\bar Z_r^{\beta, \delta, \varphi}(y)- \bar Z_r^{\beta, \delta, \varphi}(0)\right|\, \ddd y \, \ddd L^{\beta, \delta}_{r,x} \, \ddd x \, \ddd z\\
    &\qquad \to 0.
\end{align*}
Iterating through the finitely many $u_\eps$ and $v_\eps$ terms we indeed obtain that the entire term \eqref{eq:secondterm} vanishes as $\eps \to 0$.

For the last term we take a Riemann sum approximation of $Z_r^{\beta, \delta}$, i.e. taking a sequence of partitions of $[0,t]$ given by $0= t_0 \leq  \cdots \leq  t_n =t $ with mesh size $\Delta_n \to 0$ as $n \to \infty$ we consider
\begin{align*}
\sum_{i=1}^n Z^{\beta, \delta }_{t_i}(0) (L_{t_i}^{\beta, \delta }-L_{t_{i-1}}^{\beta, \delta }).
\end{align*}
Recall that we may choose the partition in a way such that
\begin{align*}
\sum_{i=1}^n Z^{\beta, \delta }_{t_i}(0) \ind_{[t_{i-1}, t_i]} (r) \to Z^{\beta, \delta }_r(0)
\end{align*}
uniformly in $r \in [0,t]$ as $n \to \infty$. 
Thus, since on a set with probability one we have $\sup_{x \in \R} L^{\beta, \delta}_{s,x}  < \infty$ for each $s \in [0,t]$ we may deduce that uniformly in $x \in \R$ using the piece-wise continuity of $r \mapsto Z_r^{\beta, \delta}$ we have  
\begin{align*}
\left\vert\sum_{i=1}^n Z^{\beta, \delta }_{t_i}(0) (L_{t_i,x}^{\beta, \delta }-L_{t_{i-1},x}^{\beta, \delta }) - \int_0^t  Z_r^{\beta,\delta}(0)\, \ddd L^{\beta, \delta}_{r,x} \right\vert\to 0.
\end{align*}
Next, we see by the triangle inequality again that 
\begin{align*}
&\left|\int_0^t \int_\R \int_\R \rho_\eps (z) \rho_\eps(z+x) Z_r^{\beta,\delta}(0)\, \ddd L^{\beta, \delta}_{r,x} \, \ddd x \, \ddd z- \int_0^t  Z_r^{\beta,\delta}(0)\, \ddd L^{\beta, \delta}_{r} \right| \\
&\quad =\left|\int_0^t \int_\R  \rho_{2\eps} (x)  Z_r^{\beta,\delta}(0)\, \ddd L^{\beta, \delta}_{r,x} \, \ddd x-  \int_0^t  Z_r^{\beta,\delta}(0)\, \ddd L^{\beta, \delta}_{r} \right| \\
& \quad \leq \left|\int_0^t \int_\R  \rho_{2\eps} (x)  Z_r^{\beta,\delta}(0)\, \ddd L^{\beta, \delta}_{r,x} \, \ddd x - \sum_{i=1}^n \int_\R\rho_{2\eps}(x) Z_{t_i}^{\beta,\delta}(0)(L^{\beta, \delta}_{t_i,x}-L^{\beta, \delta}_{t_{i-1,x}}) \, \ddd x \right|\\
& \quad \quad + \left| \sum_{i=1}^n \int_\R\rho_{2\eps}(x)Z_{t_i}^{\beta,\delta}(0) (L^{\beta, \delta}_{t_i,x}-L^{\beta, \delta}_{t_{i-1,x}}) \, \ddd x - \sum_{i=1}^n Z_{t_i}^{\beta,\delta}(0) (L^{\beta, \delta}_{t_i}-L^{\beta, \delta}_{t_{i-1}}) \right|\\
& \quad \quad + \left|  \sum_{i=1}^n  Z_{t_i}^{\beta,\delta}(0)(L^{\beta, \delta}_{t_i}-L^{\beta, \delta}_{t_{i-1}}) -  \int_0^t Z_{r}^{\beta,\delta}(0) \, \ddd L_r^{\beta ,\delta}\right|
\end{align*}
holds true. Now for any $\eta >0$ choose $n$ large enough such that 
\begin{align*}
\left|  \sum_{i=1}^n  Z_{t_i}^{\beta,\delta
}(0)(L^{\beta, \delta}_{t_i,x}-L^{\beta, \delta}_{t_{i-1},x}) -  \int_0^t Z_{r}^{\beta,\delta}(0) \, \ddd L_{r,x}^{\beta ,\delta}\right|<\eta
\end{align*}
for all $x \in \R$. Then choose $\eps >0$ small enough to get
\begin{align*}
\left| \sum_{i=1}^n \int_\R\rho_{2\eps}(x)Z_{t_i}^{\beta,\delta, \eps}(0) (L^{\beta, \delta}_{t_i,x}-L^{\beta, \delta}_{t_{i-1,x}}) \, \ddd x - \sum_{i=1}^n Z_{t_i}^{\beta,\delta, \eps}(0) (L^{\beta, \delta}_{t_i}-L^{\beta, \delta}_{t_{i-1}}) \right| < \eta
\end{align*}
using the mollifying property of the heat kernel. 
This finally yields that indeed as $\eps \to 0$ we have that the term \eqref{eq:thirdterm} also vanishes. Combining the above we obtain path-wise in $\omega$ that as $\eps \to 0$
\begin{align*}
\int_0^t \int_\R \rho_\eps (z) \rho_\eps(z+ M^\gamma_r-M^\beta_r) Z_r^{\beta,\delta, \eps}(z) \, \ddd z \, \ddd r \to \int_0^t  Z_r^{\beta, \delta} (0) \, \ddd L^{\beta, \delta}_r .
\end{align*}
It now suffices to justify the exchange of limit and expectation. For this we invoke the dominated convergence theorem. In order to find a dominating function we calculate as follows:
\begin{align*}
& \int_0^t \sum_{\beta, \delta \in I_s , \beta\neq \delta}\int_\R \rho_\eps (z) \rho_\eps(z+ M^\delta_r-M^\beta_r) Z_r^{\beta,\delta, \eps}(z) \, \ddd z \, \ddd r \\
&\quad \leq C t \sum_{\beta =1}^m \sum_{\delta =1}^m \int_0^t \int_\R\int_0^{t-r}  \rho_{2\eps} (x) \, \ddd L^{\beta, \delta}_{s,x}\, \ddd x \, \ddd r \\
&\quad \leq C  t^2 \sum_{\beta =1}^m \sum_{\delta =1}^m \sup_{x \in \R}L^{\beta, \delta}_{t,x}
\end{align*}
for some constant $C>0$.
Then again using the fact that $\sup_{x \in \R}L^{\beta, \delta}_{t,x}$ is integrable the proof is concluded.

It remains to justify the restriction to the interval $[0,\tau_m]$ in the preceding argument. To do so we modify our dual process such that it stops branching, coalescing and dying at time $\tau_m$ but may perform independent on/off Brownian motions up until time $t$,
i.e. we set for all particles $\alpha$ whenever they exist at time $\tau_m$
\begin{align*}
    \bar{M}^\alpha_t=\begin{cases}
      M^\alpha_t, & \text{if}\ t<\tau_m \\
      M^\alpha_{\tau_m}+B^\alpha_{t-\tau_m}, & \text{otherwise}
    \end{cases}
\end{align*}
where $(B^\alpha_t)_{t\geq 0}$ is an independent on/off Brownian motion started from the state of $M_{\tau_m}^\alpha$ (which is also independent of the on/off Brownian motions needed for the other particles). We denote by $\bar L_t, \bar I_t, \bar J_t, \bar K_t$ the quantities corresponding to the modified dual process $(\bar M_t)_{t\geq 0}$. We may then repeat the proof of Theorem \ref{thm:duality}. Moreover, the preceding calculations from this Lemma then show that 
\begin{align}
&\frac{1}{2}\nu\E \left[\int_0^t \int_0^{t-r}\bar K(r)  \sum_{\beta, \delta \in \bar I_r, \beta \neq \delta} \prod_{\phi \in \bar I_r \setminus \lbrace\beta, \delta \rbrace} u_\eps (s,\bar M^\phi_r) \prod_{\gamma \in \bar J_r} v_\eps (s,\bar M^\gamma_r)\right.\nonumber\\
&\qquad\qquad\qquad\qquad\qquad\qquad\qquad\times\left. \int_\R \rho_\eps(z-\bar M^\beta_r) \rho_\eps(z-\bar M^\delta_r)\sigma^2(u(s,z)) \, \ddd z \, \ddd s \, \ddd r\right] \\
&\quad \quad  \to \frac{1}{4}\nu\E\left[ \int_0^t \int_0^{t-r}  \bar K(s) \sum_{\beta, \delta \in \bar I_s, \beta \neq \delta}  \prod_{\phi \in \bar I_s\setminus \lbrace \beta, \delta \rbrace}u (r,\bar M^\phi_s) \prod_{\gamma \in \bar J_s}v (r,\bar M^\gamma_s) \right. \nonumber\\
&\qquad\qquad\qquad\qquad\qquad\qquad\times \left.  (u (r,\bar M^\delta_s)-u (r,\bar M^\beta_s)u (r,\bar M^\delta_s))  \, \ddd \bar L^{\beta,\delta}_s \, \ddd r \right] \label{eq:modifieddual}
\end{align}
as $\eps \to 0$. Note however that since the modified dual process neither branches, dies nor coalesces after time $\tau_m$ we must add the indicator $\ind_{\lbrace s \leq \tau_m\rbrace }$ to the quantities \eqref{eq:death}, \eqref{problem0} and \eqref{problem2} so that the duality relation for the modified process becomes 
\begin{align}
    \bar k(t,0,0)-\bar k(0,t,0)&=\frac{1}{4}\nu\E\left[ \int_0^t \ind_{\lbrace s \geq \tau_m\rbrace } \bar K(s) \sum_{\beta, \delta \in \bar I_s, \beta \neq \delta}  \prod_{\phi \in \bar I_s\setminus \lbrace \beta, \delta \rbrace}u (r,\bar M^\phi_s) \prod_{\gamma \in \bar J_s}v (r,\bar M^\gamma_s)\right.\nonumber\\
&\qquad\qquad\qquad\qquad\qquad\times\left.  (u (r,\bar M^\delta_s)-u (r,\bar M^\beta_s)u (r,\bar M^\delta_s))  \, \ddd \bar L^{\beta,\delta}_s \, \right]\nonumber\\
&\qquad+s\E\left[ \int_0^t  \ind_{\lbrace s \geq \tau_m\rbrace } \bar K(s) \sum_{\beta \in \bar I_s}  \prod_{\delta \in \bar I_s, \delta \neq \beta}u (r,\bar M^\delta_s)\right.\nonumber\\
&\qquad\qquad\qquad\qquad\qquad\times \left.\prod_{\gamma \in \bar J_s}v (r,\bar M^\gamma_s) (u ^2(r,\bar M^\beta_s)-u (r,\bar M^\beta_s))  \, \ddd s  \right]\nonumber\\
&\qquad +m_2\E\left[ \int_0^t\ind_{\lbrace s \geq \tau_m\rbrace } \bar K(s) \sum_{\beta \in \bar I_s}  \prod_{\delta \in \bar I_s, \delta \neq \beta}u (r,\bar M^\delta_s)\right.\nonumber\\
&\qquad\qquad\qquad\qquad\qquad\times \left.\prod_{\gamma \in \bar J_s}v (r,\bar M^\gamma_s) (1-u (r,\bar M^\beta_s))  \, \ddd s  \right]. \label{eq:modifieddual2}
\end{align}
The second term may be bounded for some $C>0$ by
\begin{align*}
    C \E\left[\int_0^t \vert \bar I_s \vert \ind_{\lbrace s \geq \tau_m\rbrace } \, \ddd s \right]
\end{align*}
which vanishes since $\ind_{\lbrace s \geq \tau_m\rbrace } \to 0$ as $m \to \infty$ almost surely and $\vert \bar I_s \vert$ can be dominated by a Yule process with rate $s$ uniformly in $m$. A similar calculation yields the same conclusion for the third term. For the first term we have the bound 
\begin{align*}
    C \E\left[\int_0^t \ind_{\lbrace s \geq \tau_m\rbrace } \sum_{\beta, \delta \in \bar I_s, \beta \neq \delta} \, \ddd L^{\beta, \delta}_s  \right] &\leq C\E\left[ \sum_{\beta, \delta \in \bar I_t, \beta \neq \delta} (L^{\beta, \delta}_t -L^{\beta, \delta}_{\tau_m})^+\right].
\end{align*}
Now, since there is no more branching, coalescence and death after time $\tau_m$ the expected intersection local time per pair of particles after time $\tau_m$ is dominated by the expected local time of a single \textit{standard} Brownian motion at 0 up until time $\sqrt{2}t$ which we denote by $L$. Hence, denoting by $(\mathcal{G}_t)_{t \geq 0}$ the natural filtration of the dual process we even have that the first term is bounded by
\begin{align*}
     C\E\left[ \sum_{\beta, \delta \in \bar I_t, \beta \neq \delta} \ind_{\lbrace \tau_m \leq t \rbrace} \E[L^{\beta, \delta}_t-L^{\beta, \delta}_{\tau_m}\vert \mathcal{G}_{\tau_m}]\right] \leq C\E\left[\vert \bar I_t \vert^2 \ind_{\lbrace \tau_m \leq t \rbrace} \E[L_{\sqrt{2}t}]\right].
\end{align*}
This quantity also vanishes as $m \to \infty$. Thus, we obtain that the right hand side of Equation \eqref{eq:modifieddual2} converges to $0$ as $m \to \infty$. For the left hand side note that the modified dual converges to the original dual process as $m \to \infty$ and invoke the dominated convergence theorem. 
\end{proof}

\noindent As a corollary, the moment duality of Theorem \ref{thm:duality}
allows us to infer uniqueness in law for the SPDE \eqref{eq:SteppingStoneWithSeedbank}. 
Note that the underlying arguments are standard but 
some care is needed due to the fact that we allow for non-continuous initial conditions.

\begin{proof} [Proof of Theorem \ref{thm:Uniqueness in law}]

Recall that if $(u,v)$ is a solution of the system $\eqref{eq:SteppingStoneWithSeedbank}$ with initial conditions $v_0,u_0 \in B(\R,[0,1])$, then 
\begin{align*}
    \tilde u(t,x)& :=  u(t,x) -\hat u_0(t,x) := u(t,x) - \int_\R G(t,x,y) u_0(y)\, \ddd y,\\
    \tilde v(t,x)& := v(t,x) - \hat v_0(t,x) := v(t,x) - e^{-c't}v_0(x) 
\end{align*}
have paths taking values in $C([0,\infty[,C(\R))$, see Remark \ref{rem_continuity} and \eqref{variation of constants}. 
Now, by Theorem \ref{thm:duality} the law of the dual process uniquely determines the mixed moments of
\begin{align*}
    \left((\tilde u(t,x_1),\tilde v(t,y_1)),\ldots, (\tilde u(t,x_n),\tilde v(t,y_n))\right)
\end{align*}
and hence, by uniqueness in the Hausdorff moment problem, 
also the joint distribution of the above quantity for each fixed $t\ge0$ and  arbitrarily chosen $x_1,y_1,\ldots ,x_n,y_n \in \R$, $n\in\N$.
Since the cylindrical $\sigma$-algebra on $C(\R)$ coincides with the Borel-$\sigma$-algebra w.r.t.\ the topology of locally uniform convergence, this shows that the distribution of $(\tilde u(t,\cdot),\tilde v(t,\cdot))$ on $C(\R)^2$ is uniquely determined for each fixed $t\ge0$. In order to extend this to the finite-dimensional distributions, one can use the martingale problem corresponding to the weak formulation of the equation. Using the well-known fact that uniqueness of the one-dimensional time marginals in a martingale problem implies uniqueness of the finite-dimensional time marginals, we then obtain that the distribution of $(\tilde u,\tilde v)$ on $C([0,\infty[,C(\R))^2$ is uniquely determined. 
Finally, this implies that also the distribution of $( u, v)=(\tilde u+\hat u_0,\tilde v+\hat v_0)$ on $C\left(]0,\infty[,C(\R,[0,1])\right)\times C\left([0,\infty[,B(\R,[0,1])\right)$ is uniquely determined.

\end{proof}



\section{An application to the F-KPP Equation with seed bank}
\label{sec:F-KPP}

\noindent We are interested in applying the previously established results to the F-KPP Equation with seed bank, i.e.\ the system:
\begin{align} \label{eq:Fkpp_seedbank}
    \partial_t  p(t,x) &= \frac{1}{2} \Delta p(t,x) + (1-p(t,x))p(t,x) + c( q(t,x)-  p(t,x)),\nonumber\\
    \partial_t q(t,x) &=  c'(p(t,x)- q(t,x))
\end{align}
with $p_0=q_0 =\ind_{] -\infty,0]}$.
This means that in our original equation we set $m_1=m_2=\nu=0,s=1$ and return to the setting $p =1-u$. We do this to make the results of this section more easily comparable with the literature concerning the original F-KPP Equation and avoid confusion (see e.g. \cite{F37,McK75}). This also implies that the dual process is now ``merely" an on/off branching Brownian motion.
Note that here the selection term is always positive for $p$ implying that it corresponds to the beneficial type. \\ Recall for the case of the classical F-KPP Equation that since we start off with a Heaviside initial condition concentrated on the negative half axis, the wave speed $\sqrt 2$ also becomes the asymptotic speed at which the beneficial allele ``invades" the positive half axis. \\
\textit{In this section we consider the question to what extent the introduction of the seed bank influences the invasion speed of the beneficial allele.}\\
We begin by proposing a formal definition of the invasion speed.
\begin{defn}
In the setting of Equation \eqref{eq:Fkpp_seedbank}, with Heaviside initial conditions, we call $\xi \geq 0$ the asymptotic invasion speed of the beneficial allele if 
\begin{align*}
    \xi = \inf \big\lbrace  \lambda \geq 0 \big\vert \lim_{t \to \infty} p (t,\lambda t) =0\big\rbrace.
\end{align*}
\end{defn}

\noindent In the case of the classical F-KPP Equation we have $\xi=\sqrt{2}$. Intuitively, one would expect this speed to be reduced in the presence of a seed bank.

\noindent In order to investigate this we aim to employ the duality technique established in the preceding section. Recall that in the setting of Corollary \ref{corol:dualityF-KPP} the duality is given by $$p(t,x)=1-\PP_{(0,\boldsymbol{a})}\left(\max_{\beta \in I_t \cup J_t} M_t^\beta \leq x \right)$$
where $(M_t)_{t\geq 0}$ is an on/off branching Brownian motion, $I_t,J_t$ are the corresponding index sets of active and dormant particles at time $t\geq 0$, respectively, and $\PP_{(x,\boldsymbol{a})}$ is the probability measure under which the on/off BBM is started from a single \textit{active} particle at $x\in\R$. This clearly resembles the duality from the classical FKPP Equation given by 
$$p(t,x)=1-\PP_0\big(\max_{\beta \in I_t} \tilde B_t^\beta \leq x \big)$$
where $( \tilde B_t)_{t\geq 0}$ is a simple branching Brownian motion.

\begin{proof} [Proof of Corollary \ref{corol:dualityF-KPP}]

By Theorem \ref{thm:duality}  for initial values $(x_1,\ldots, x_n) \in \R^n$ and $(y_1, \ldots , y_m) \in \R^m$ we have that
\begin{align*}
    \E\left[ \prod_{i=1}^n (1- p(t,x_i)) \prod_{i=1}^m (1-q(t,y_i)) \right] = \E\left[ \prod_{\beta \in I_t} (1- p_0(B_t^\beta)) \prod_{\gamma \in J_t} (1-q_0(B_t^\gamma)) \right].
\end{align*}

\noindent Plugging in our specific initial conditions and using that the solution $(p , q)$ is deterministic we see
\begin{align*}
    1-p(t,x) &=\PP_{(x,\boldsymbol{a})}\left(\min_{\beta \in I_t \cup J_t} M_t^\beta \geq 0 \right)\\
     &=\PP_{(0,\boldsymbol{a})}\left(\min_{\beta \in I_t\cup J_t} M_t^\beta +x\geq 0 \right)\\
     &=\PP_{(0,\boldsymbol{a})}\left(\max_{\beta \in I_t \cup J_t} -M_t^\beta \leq x \right)=\PP_{(0,\boldsymbol{a})}\left(\max_{\beta \in I_t \cup J_t} M_t^\beta \leq x \right).
\end{align*}

\end{proof}
\noindent Next, we establish an upper bound for the asymptotic speed of the rightmost particle. 
Writing $K_t\defeq I_t \cup J_t$ for $t \geq 0$ and denoting by $B=(B_t)_{t \geq 0}$ an on/off BM without branching,
we have
\begin{align*}
    \PP( \exists \beta \in K_t\colon \, M^\beta_t >\lambda t  ) &\leq \E \left[ \sum_{\beta \in K_t} \ind_{\lbrace M^\beta_t>\lambda t \rbrace} \right]\\
    &= \E\left[\,\abs{K_t}\right] \PP(B_t>\lambda t )
\end{align*}
where we have used the following simple many-to-one lemma:
\begin{lemma}
For any $t\geq 0$ and any measurable function $F\colon \R \to \R$ we have 
\begin{align*}
    \E_{(x,\sigma)} \left[ \sum_{\beta \in K_t} F(M^\beta_t)\right] = \E_{(x,\sigma)}\left[\,\abs{K_t}\right] \E_{(x,\sigma)}\left[F(B_t) \right]
\end{align*}
where $B$ is a on/off Brownian motion under $\PP_{(x,\sigma)}$ starting in a single particle with initial state $\sigma \in \lbrace\boldsymbol a, \boldsymbol d \rbrace$ and initial position $ x \in \R$.
\end{lemma}
\begin{proof}
Using that the number of particles is independent of the movement of the active particles we get
\begin{align*}
    \E_{(x,\sigma)} \left[ \sum_{\beta \in K_t} F(M^\beta_t)\right] &= \sum_{n=1}^\infty \E_{(x,\sigma)} \left[ \sum_{\beta =1}^n F(M^\beta_t) \ind_{\lbrace \abs{K_t}=n \rbrace}\right] \\
    &=\sum_{n=1}^\infty \E_{(x,\sigma)} \left[ \sum_{\beta =1}^n F(M_t^\beta ) \right] \PP_{(x,\sigma)}(\, \abs{K_t}=n)\\
    &=\sum_{n=1}^\infty n\E_{(x,\sigma)}[F(B_t)]\PP_{(x,\sigma)}(\, \abs{K_t}=n)=\E_{(x,\sigma)}[\, \abs{K_t}] \E_{(x,\sigma)}[F(B_t)].
\end{align*}
This proves the result.
\end{proof}
\noindent Now, we first compute $\E_{(0, \boldsymbol{a})}\left[\,\abs{K_t}\right]$. Note that $(\abs{I_t}, \abs{J_t})_{t \geq 0}$ is a continuous time discrete state space Markov chain on $\N_0 \times \N_0$ with the following transition rates:
\begin{align*}
\text{birth}&\colon
i\to i+1, \text{ with rate } i\\
\text{active to dormant}&\colon
\begin{cases}
i\to i-1, \\
j\to j+1
\end{cases} \text{ with rate } ci\\
\text{dormant to active}&\colon
\begin{cases}
i\to i+1, \\
j\to j-1.
\end{cases}\text{ with rate } c'j
\end{align*}
For the expectations we then get the following system of ODE's for $x=x(t)=\E_{(0, \boldsymbol{a})}\left[\,\abs{I_t} \right]$ and $y=y(t)=\E_{(0, \boldsymbol{a})}\left[\,\abs{J_t} \right]$ (cf. \cite[V.7]{AN72} )
\begin{align*}
    x'&=x-cx+c'y ,\\
    y'&=cx-c'y.
\end{align*}
With the initial condition $(x(0),y(0))=(1,0)$ we obtain the following closed form solution:
\begin{align*}
    x(t)&= \frac{1}{\sqrt{a}} \left(\frac{c'-c+\sqrt{a}+1}{2}\right) \exp\left( \left(\frac{-c-c'+1+\sqrt{a}}{2 }\right)t\right) \\
    &\quad - \frac{1}{\sqrt{a}} \left(\frac{c'-c-\sqrt{a}+1}{2}\right) \exp\left( -\left(\frac{c+c'-1+\sqrt{a}}{2 }\right)t\right),\\
    y(t) &= \frac{1}{\sqrt{a}} c \exp\left( \left(\frac{-c-c'+1+\sqrt{a}}{2 }\right)t\right)- \frac{1}{\sqrt{a}} c \exp\left( -\left(\frac{c+c'-1+\sqrt{a}}{2 }\right)t\right)
\end{align*}
abbreviating $a:=(c-1)^2 +2cc' + (c')^2 +2c'$.\\
\noindent Finally, we aim to control $\PP_{(0,\boldsymbol{a})}(B_t >\lambda t)$. To this end we recall the well known tail bound for the normal distribution given by 
\begin{align}\label{eq:tailbound}
    \frac{1}{\sqrt{2\pi}} \int_x^\infty e^{-y^2/2} \, \ddd y \leq \frac{e^{-x^2/2}}{x \sqrt{2 \pi}}
\end{align}
for $x \geq 0$.


\noindent To employ this, note first that $\PP_{(0,\boldsymbol{a})}(B_t> \lambda t )$ is equivalent to
\begin{align*}
    \PP_{(0,\boldsymbol{a})}(\tilde B_{t-X_t}> \lambda t )
\end{align*}
where $X_t$ is the amount of time the on/off Brownian path $(B_t)_{t \geq 0}$ is switched off until time $t\geq 0$ and $\tilde B =(\tilde B_t)_{t \geq 0}$ is a \textit{standard} Brownian motion started at $0$. By independence we then have 
\begin{align*}
    \PP_{(0,\boldsymbol{a})}( \exists \beta \in K_t\colon \, M^\beta_t >\lambda t  ) 
    &\leq \E_{(0,\boldsymbol{a})}\left[\,\abs{K_t}\right] \PP_{(0,\boldsymbol{a})}(\tilde B_{t-X_t}>\lambda t )\\
    &= \E_{(0,\boldsymbol{a})}\left[\,\abs{K_t}\right] \E_{(0,\boldsymbol{a})}\left[\E\left[ \ind_{\lbrace \tilde  B_{t-X_t}>\lambda t \rbrace}\big\vert X_t \right]\right]\\
    &= \E_{(0,\boldsymbol{a})}\left[\,\abs{K_t}\right] \E_{(0,\boldsymbol{a})}\left[\PP ( \tilde B_{t-s}>\lambda t )\big\vert_{s=X_t} \right].
\end{align*}
Then, using Equation \eqref{eq:tailbound} we get for $s<t$
\begin{align*}
    \PP (\tilde  B_{t-s}>\lambda t ) &= \PP(\sqrt{t-s} \tilde B_1 > \lambda t)=\PP(\tilde B_1> \lambda t /\sqrt{t-s} ) \\
        &\leq \frac{1}{\sqrt{2 \pi} \frac{\lambda t}{\sqrt{t-s}}} e^{-\frac{\lambda ^2 t ^2}{2(t-s)}} =\frac{\sqrt{t-s}}{\sqrt{2 \pi}\lambda t} e^{-\frac{\lambda ^2 t ^2}{2(t-s)}} \\
        &\leq \frac{1}{\sqrt{2 \pi} \lambda \sqrt{t}} e^{-\frac{\lambda ^2 t}{2}} .
\end{align*}
Thus, we may complete our calculation in the following manner:
\begin{align*}
    \PP_{(0,\boldsymbol{a})}( \exists \beta \in I_t\colon \, M^\beta_t >\lambda t  ) &\leq \left(\frac{1}{\sqrt{a}} \left(\frac{c'-c+\sqrt{a}+1}{2}\right) \exp\left( \left(\frac{-c-c'+1+\sqrt{a}}{2 }\right)t\right) \right.\\
    &\quad\quad - \left.\frac{1}{\sqrt{a}} \left(\frac{c'-c-\sqrt{a}+1}{2}\right) \exp\left( -\left(\frac{c+c'-1+\sqrt{a}}{2 }\right)t\right)\right.\\
    &\quad\quad + \left. \frac{1}{\sqrt{a}} c \exp\left( \left(\frac{-c-c'+1+\sqrt{a}}{2 }\right)t\right)\right.\\
   & \left. \quad\quad - \frac{1}{\sqrt{a}} c \exp\left( -\left(\frac{c+c'-1+\sqrt{a}}{2 }\right)t\right) \right) \frac{e^{-\lambda^2t /2}}{\lambda \sqrt{t}\sqrt{2 \pi}}\\
   &\to 0
\end{align*}
as $t \to \infty$ for $\lambda^2 \geq -c'-c+\sqrt{a}+1$.  
Setting $\lambda^*\defeq\sqrt{-c'-c+\sqrt{a}+1}$, the preceding calculations show that for every $\varepsilon >0$ there exist some constants $C, C(\varepsilon)>0$ depending on $c,c'$ such that for $t$ large enough we have the bound
\begin{align*}
    \PP_{(0,\boldsymbol{a})}\left(R_t/t>(1+\varepsilon)\lambda^*\right) \leq C \exp(-C(\varepsilon)t).
\end{align*}
Then an application of the Borel-Cantelli Lemma shows that we have proved the following proposition:
\begin{prop} \label{prop:NewBoundWaveSpeed}
Denote by $(R_t)_{t\geq 0}$ the process describing the position of the rightmost particle of the on/off branching Brownian motion with switching rates $c,c'\geq 0$ starting from a single \textit{active} particle at $0$. Then we have almost surely that
\begin{align*}
    \limsup_{t \to \infty} \frac{R_t}{t} \leq \lambda 
\end{align*}
for every $\lambda$ such that $\lambda^2 \geq -c'-c+\sqrt{a}+1$, i.e.\ the speed of the rightmost particle is asymptotically bounded by $ \sqrt{-c'-c+\sqrt{a}+1}\, t$, where $a=(c-1)^2 +2cc' +(c')^2 +2c'$.
\end{prop}

\begin{figure}[h]
    \centering
    \scalebox{0.7}[0.5]{\includegraphics{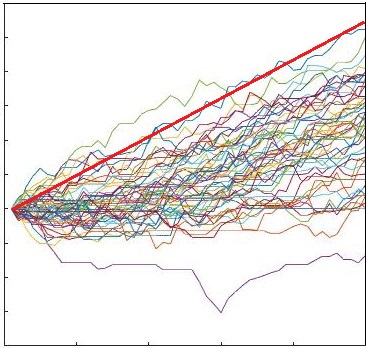}}
    \caption{Independent realizations of the trajectories of the rightmost particle of $(\tilde M_t)_{t \geq 0}$ plotted against the asymptotic speed (red).}
    \label{fig:MaximumparticleSpeed}
\end{figure}

\begin{remark}
Judging from simulations (cf. Figure 2) the upper bound does not seem unreasonable. It is interesting to note that the value of $\lambda$ is entirely specified by the expected number of particles, as we may use the same tail bounds for $\PP(B_t > \lambda t)$ and $\PP(M_t >\lambda t)$. \\
However, establishing a matching lower bound seems to be much more challenging and is currently beyond the scope of this paper. We hope to tackle this question in future work.
\end{remark}
\noindent Finally, we apply our deductions for the dual process to our original PDE. Via the duality relation from Corollary \ref{corol:dualityF-KPP} we get:
\begin{corol} \label{corol: propagationspeed}
In the case of the F-KPP Equation with seed bank we have that
\begin{align*}
    \xi\leq \sqrt{-c'-c+\sqrt{a}+1}.
\end{align*}
\end{corol}

\begin{proof}
This follows directly from the fact that
 \begin{align*}
     p(t,\lambda t) &= \PP_{(0,\boldsymbol{a})}(R_t > \lambda t) = \PP_{(0,\boldsymbol{a})}( \exists \beta \in K_t\colon \, M^\beta_t >\lambda t  ).
 \end{align*}
\end{proof}
\noindent This shows that indeed the invasion speed for $c=c'=1$ is slowed down from $\sqrt{2}$ to (at least) $\sqrt{\sqrt{5}-1} \approx 1.111$.

\begin{remark}
Note that for the boundary case $c=c'\to 0$ we recover the upper bound $\sqrt{2}$ from the classical F-KPP Equation. For $c,c' \to \infty$ the upper bound becomes $\sqrt{1}$ showing that instantaneous mixing with the dormant population leads to a slow down from the classical F-KPP Equation corresponding to a doubled effective population size. On the level of the dual process this could be interpreted as  essentially halving diffusivity and branching rate.
\end{remark}
\section{Proofs for Section 2}\label{sec:ProofForSection2}

\begin{proof} [Proof of Proposition \ref{thm:SIE_Representation}]
The proof is standard and follows along the lines of \cite[Thm.\ 2.1]{S94} which is why we only provide a rough outline. \\
Suppose that for each fixed $(t,y)\in\,]0,\infty[\times\R$, Equation \eqref{SIE} is satisfied almost surely. Then using the linear growth bound and $L^2$-boundedness, a simple Fubini argument shows that almost surely Equation \eqref{SIE} is satisfied for almost all $(t,y)\in\,]0,\infty[\times\R$. Thus given $\phi\in C_c^\infty(\R)$, we can plug \eqref{SIE} (with $s$ in place of $t$) into $\int_0^t\langle u(s,\cdot),\Delta\phi\rangle\,\ddd s$. Then a straightforward but tedious calculation
using Fubini's theorem for Walsh's stochastic integral (cf.\ Theorem 5.30 in \cite{K09}) 
shows that Equation \eqref{spdeschwartz} holds almost surely for each fixed $t\ge0$. Since both sides are continuous in $t\ge0$, \eqref{spdeschwartz} holds for all $t\ge0$, almost surely.
Thus $(u,v)$ is also a solution to \eqref{eq: SPDE_general_shortnotation} in the sense of Definition \ref{def: weaksolution}.\\
Conversely, suppose that for each $\phi\in C_c^\infty(\R)$, Equation \eqref{spdeschwartz} holds for all $t\ge0$, almost surely.\\
\textit{Step 1:} Defining 
\begin{align*}
    C_{\text{rap}}(\R) :=\left\{ f \in C(\R) \,\middle\vert\, \|f\|_\lambda\defeq\sup_{x \in \R} e^{\lambda \abs{x}}\abs{f(x)}< \infty \text{ for all } \lambda >0 \right\}
\end{align*}
and
\begin{align*}
    C^2_{\text{rap}}(\R) :=\left\lbrace f \in  C_{\text{rap}}(\R)\, \big\vert\, f', f''\in C_{\text{rap}}(\R) \right\rbrace,
\end{align*}
one can show that Equation \eqref{spdeschwartz} also holds for each $\psi \in C^2_{\text{rap}}(\R)$.

\noindent\textit{Step 2:} For $T>0$, define $C^{(1,2)}_{T,\text{rap}}$ as the space of all functions $f:[0,T[\times\R\to \R$ such that $ t\mapsto f(t,\cdot)$ is a continuous $C^2_{\text{rap}}$-valued function and $t\mapsto\partial_tf(t,\cdot)$ is a continuous $C_{\text{rap}}$-valued function of $t\in[0,T[$.
Then one can show that for 
each $\phi\in C^{(1,2)}_{T,\text{rap}}$,
we have that an integration-by-parts like equation holds, i.e.\ for all fixed $ t\in[0,T[$ we have almost surely
\begin{align}\label{integration-by-parts}
    &\langle u(t, \cdot ), \phi(t, \cdot ) \rangle - \langle u(0, \cdot ), \phi(0, \cdot ) \rangle \\
    &\qquad\qquad =\int_0^t \langle u(s, \cdot ),  \partial_s \phi(s, \cdot )\rangle  +\langle u(s, \cdot ),\frac{\Delta}{2}  \phi(s, \cdot ) \rangle  +\langle b(s,\cdot,u(s, \cdot ),v(s,\cdot)), \phi(s, \cdot ) \rangle \, \ddd s\nonumber\\
    &\qquad\qquad \qquad +\int_0^t \int_\R \sigma(s,x,u(s,x),v(s,x)) \phi(s,x) \,  W(\ddd x,\ddd s).\nonumber
\end{align}

\noindent\textit{Step 3:} Fix $T>0$, $y\in\R$ and define the time-reversed heat kernel $\phi_{T}^y(t,x):=G(T-t,x,y)$ for $ t\in[0,T[$ and $x\in\R$.
Then we have $\phi_{T}^y\in C^{(1,2)}_{T,\text{rap}}$,
and plugging it into \eqref{integration-by-parts} 
in place of $\phi$ and using that the heat kernel solves the heat equation, we obtain the required result upon taking the limit as $t\uparrow T$. 
\end{proof}

\begin{proof}[Proof of Theorem \ref{Lipschitzexistence}]
The proof follows along the lines of \cite[Thm.\ 2.2]{S94} and only has to be adapted to the two component case.\\

\noindent\textit{Step 1:} Fix $p\ge2$. Define $u_1(t,y)\defeq\langle u_0, G(t,\cdot , y) \rangle$ and $v_1(t,y)\defeq v_0(y)$ for all $(t,y) \in [0,\infty[\,\times \R$ and inductively
\begin{align*}
    u_{n+1}(t,y) & \defeq \langle u_0, G(t,\cdot , y) \rangle +  \int_0^t \langle  b(s,\cdot, u_n(s,\cdot ), v_n(s,\cdot)), G(t-s, \cdot , y ) \rangle \, \ddd s \\
    & \quad  + \int_\R \int_0^t \sigma(s,x, u_n(s,x), v_n(s,x)) G(t-s,x,y) \, W(\ddd s, \ddd x),\\
    v_{n+1}(t,y) & \defeq v_0(y) + \int_0^t \tilde b(s,y,u_n(s,y),v_n(s,y)) \, \ddd s
\end{align*}
for all $n\in\N$.
Then one can use the linear growth bound \eqref{eq: lineargrowth} to obtain that 
this defines a sequence such that for every $n \in \N$ and $T>0$
\begin{align*}
    \left\|(u_n,v_n)\right\|_{T,p} =  \sup_{0\leq t \leq T}\sup_{x \in \R} \E \left[ \left(\abs{u_n(t,x)}+\abs{v_n(t,x)}\right)^p \right]^{1/p}< \infty.
\end{align*}
Similarly, using the Lipschitz condition \eqref{eq:Lipschitzcondition} instead of the linear growth bound one can obtain that for all $n \in \N$ and $t\in[0,T]$
\begin{equation*}
    \left\|(u_{n+1}, v_{n+1})-(u_{n}, v_{n})\right\|_{t,p} ^{ p} \leq C_{T}   \int_0^t (t-s)^{-\frac{1}{2}}\left\|(u_{n}, v_{n})-(u_{n-1}, v_{n-1})\right\|_{s,p}^{ p} \, \ddd s .
\end{equation*}
Applying H\"olders inequality for some $ q>1$ with conjugate $\frac{q}{q-1}\in\, ]1,2[$, we get 
\begin{equation*}
    \left\|(u_{n+1}, v_{n+1})-(u_{n}, v_{n})\right\|_{t,p} ^{ pq} \leq C_{T,p,q}   \int_0^t \left\|(u_{n}, v_{n})-(u_{n-1}, v_{n-1})\right\|_{s,p}^{ pq} \, \ddd s 
\end{equation*}
for 
 all $n \in \N$, $t\in[0,T]$, $T>0$.
By a version 
of Gronwall's lemma (see e.g.\ \cite[Lemma 6.5]{K09}), this implies
\begin{equation*}
    \sum_{n=0}^\infty \left\|(u_{n+1}, v_{n+1})-(u_{n}, v_{n})\right\|_{t,p}  < \infty.
\end{equation*}
Thus $(u_{n}(t,y), v_{n}(t,y))_{n\in\N}$ is a Cauchy sequence in $L^p$ for each fixed $t\geq 0$ and $y \in \R$,
and we can define a predictable random field $(u(t,y),v(t,y))_{t\ge0,y\in\R}$ as the corresponding limit in $L^p$.
Clearly we have
\begin{equation*}
    \|(u_n,v_n)-(u,v)\|_{T,p} \to 0
\end{equation*}
as $n \to \infty$ for all $T>0$, and $(u,v)$ satisfies Equation \eqref{SIE} and \eqref{spde_v} almost surely for each fixed $t\ge0$ and $y\in\R$.
Up to now, $p\ge2$ was fixed, but since convergence in $L^{p'}$ implies convergence in $L^p$ for $p'>p$,  
$(u,v)$ is actually $L^p$-bounded for all $p\ge2$ in the sense of Definition \ref{def: weaksolution}.\\

\noindent\textit{Step 2:} 
We argue that $u$ constructed in \textit{Step 1} has a modification $\tilde u$ with paths in $C(]0,\infty[,C(\R))$, which is equivalent to the random field $\tilde u$ being jointly continuous in $(t,x)\in\ ]0,\infty[\, \times\R$.

Consider first the stochastic integral part $$X(t,y) \defeq \int_\R \int_0^t \sigma(s,x, u(s,x),v(s,x)) G(t-s,x,y) \,  W (\ddd s,\ddd x).$$ Then for each $p\ge1$ 
and $ t,r\in[0, T]$, using the Burkholder-Davis-Gundy and Jensen inequalities as well as the linear growth bound and $L^{2p}$-boundedness of $(u,v)$ from \textit{Step 1}, we obtain that
 \begin{align*}
     \E\left[ \abs{X(t,y)-X(r,z)}^{2p}\right] & = \E \bigg[\bigg|\int_0^{t \vee r} \int_\R (G(t-s,x,y)-G(r-s,x,z)) \\
     &\qquad \times \sigma(s,x,u(s,x),v(s,x)) \, W (\ddd s,\ddd x)\bigg|^{2p}\bigg] \\
     &\leq C(p)\, \E \bigg[\bigg(\int_0^{t \vee r} \int_\R (G(t-s,x,y)-G(r-s,x,z))^{2}\\
     &\qquad \times \abs{\sigma(s,x,u(s,x),v(s,x))}^{2} \, \ddd x \, \ddd s \bigg)^p\bigg] \\
     &\leq C_{p,T} \left( \int_0^{t \vee r} \int_\R (G(t-s,x,y)-G(r-s,x,z))^{2} \, \ddd x \, \ddd s\right)^{p}\\
     &\qquad \times \left(1+\sup_{x \in \R} \sup_{0\leq t \leq T} \E\left[ \left(\vert u(t,x) \vert +\vert v(t,x)\vert \right)^{2p}\right] \right) \\
     & \leq C_{p,T} \left(\,\abs{t-r}^{p/2} +\abs{y-z}^p\right). 
 \end{align*}
Here, for the last inequality we have also used that by the calculation in the proof of \cite[Theorem 6.7]{K09} we have
\begin{equation*}
    \int_0^{t \vee r} \int_\R (G(t-s,x,y)-G(r-s,x,z))^{2} \, \ddd x \, \ddd s \leq C (\,\abs{t-r}^{1/2} +\abs{y-z}).
\end{equation*}
Next, consider the term 
 \begin{align*}
     Y(t,y) \defeq \int_0^t \int_\R G(t-s,x,y) b(s,x,u(s,x),v(s,x)) \, \ddd x \ddd s
 \end{align*}
 for $t \geq 0, y\in \R$. Let $0\le r<t\le T$.
To obtain similar estimates as for $X$ we first split the difference into two terms:
\begin{align*}
    &\E\left[ \vert Y(t,y)-Y(r,z) \vert ^{2p}\right]\\
    &\qquad \leq C(p)\,\E\Bigg[ \left(\int_r^{t} \int_\R \vert G(t-s,x,y)  b(s,x,u(s,x),v(s,x))\vert \, \ddd x \, \ddd s\right)^{2p}\Bigg]\\
    &\qquad \qquad + C(p)\,\E\Bigg[ \left(\int_0^{r} \int_\R \vert G(t-s,x,y)-G(r-s,x,z)\vert \,\vert  b(s,x,u(s,x),v(s,x))\vert \, \ddd x \, \ddd s\right)^{2p}\bigg].
\end{align*}

\noindent By Jensen's inequality, the linear growth bound and $L^{2p}$-boundedness of $(u,v)$, we have for the first term on the right hand side that
\begin{align*}
    &\E\Bigg[\left(\int_r^t \int_\R  G(t-s,x,y) \vert b(s,x,u(s,x),v(s,x))\vert \, \ddd x \ddd s \right)^{2p}\Bigg] \\
    &\qquad \leq C_{p,T}\left(\int_r^t \int_\R  G(t-s,x,y)  \, \ddd x \ddd s \right)^{2p}\left(1+\sup_{x \in \R} \sup_{0 \leq t \leq T} \E\left[ \left(\vert u(t,x) \vert +\vert v(t,x)\vert \right)^{2p}\right]\right)\\
    & \qquad \leq C_{p,T}\, \vert t-r\vert ^{2p}.
\end{align*}
For the second term we proceed analogously, using in addition \cite[Lemma 5.2]{M03} (choose $\beta =\frac{1}{2}$ and $\lambda'=0$ there), to obtain
\begin{align*}
   &\E\bigg[\left(\int_0^r \int_\R  \vert G(t-s,x,y)-G(r-s,x,z) \vert \, \vert b(s,x,u(s,x),v(s,x))\vert \, \ddd x \ddd s \right)^{2p}\bigg]\\
   &\qquad \leq C_{p,T}\left(\int_0^r \int_\R  \vert G(t-s,x,y)-G(r-s,x,z) \vert \, \ddd x \ddd s \right)^{2p} \\
   &\qquad \qquad \times\left(1+\sup_{x \in \R} \sup_{0 \leq t \leq T} \E\left[ \left(\vert u(t,x) \vert +\vert v(t,x)\vert \right)^{2p}\right]\right)\\
   &\qquad \leq C_{p,T} \left(\int_0^r   (r-s)^{-1/2} \vert t-r\vert^{1/2} + (r-s)^{-1/4} \vert y-z\vert^{1/2} \, \ddd s \right)^{2p} \\
   &\qquad \leq C_{p,T} \left(\vert t-r\vert^{ p} +\vert y-z\vert^{ p}\right)
\end{align*}
for all $0\le r<t\le T$. 
Combining the above, we have shown for $Z:=X+Y$ that 
\begin{equation}\label{bound_kolmogorov}
\E\left[|Z(t,y)-Z(r,z)|^{2p}\right]\le C_{p,T}\left(|t-r|^{p/2}+|y-z|^p\right)
\end{equation}
for all $t,r\in[0,T]$, $y,z\in\R$ and $p\ge1$.
Choosing $p>4$, we see that indeed the conditions of Kolmogorov's continuity theorem (see e.g.\ \cite[Thm.\ 3.23]{K02}) are satisfied. Hence, $Z$ has a modification jointly continuous in $(t,x)\in[0,\infty[\times\R$. \\
Now observing that 
\begin{align*}
  (t,y) \mapsto u(t,y)-Z(t,y)=\int_\R G(t,x,y) u_0(x) \, \ddd x
\end{align*}
is continuous on $]0,\infty[\times\R$, we are done. 
Note however that we cannot extend this to $t=0$ if the initial condition $u_0$ is non-continuous.\\

\noindent\textit{Step 3:} Let us write $\tilde u$ for the continuous modification of $u$ from \textit{Step 2}. 
Given this $\tilde u$ we consider path-wise in $\omega$ for each $y \in \R$ the 
integral equation 
\begin{align}\label{IE}
     \tilde v(t,y) = v_0(y)+\int_0^t\tilde b(s,y,\tilde u(s,y),\tilde v(s,y))\,\ddd s,\qquad t\ge0.
\end{align}
Since $(s,\tilde v)\mapsto\tilde b(s,y,\tilde u(s,y),\tilde v)$ is locally bounded, 
the Lipschitz condition \eqref{eq:Lipschitzcondition} implies that \eqref{IE} has a unique solution $\tilde v(\cdot,y)$. Note that since $\tilde v$ is defined path-wise, almost surely it satisfies \eqref{IE} for all $t\ge0$ and $y\in\R$. In order to see that $\tilde v$ has paths taking values in $C([0,\infty[,B_{\mathrm{loc}}(\R))$, it suffices to show that it is locally bounded in $(t,y)\in[0,\infty[\times\R$, which follows from the linear growth bound by a simple Gronwall argument. 

\noindent Now using the Lipschitz condition again, it is easy to see that $\tilde v$ is a modification of $v$ constructed in \textit{Step 1} and that  for each fixed $t>0$ and $y\in\R$, $(\tilde u,\tilde v)$ satisfies Equation \eqref{SIE} almost surely.
By Proposition \ref{thm:SIE_Representation} we conclude that $(\tilde u, \tilde v)$ is a solution of Equation \eqref{eq: SPDE_general_shortnotation} in the sense of Definition \ref{def: weaksolution}.
\\

\noindent\textit{Step 4:} It remains to show uniqueness:
Let $(u,v), (h,k)\in C(]0,\infty[,C(\R))\times C([0,\infty[,B_{\mathrm{loc}}(\R))$ be two $L^2$-bounded (strong) solutions to Equation $\eqref{eq: SPDE_general_shortnotation}$ 
in the sense of Definition \ref{def: weaksolution} with the same initial conditions $u_0,v_0\in B(\R)$. 
Then using Proposition \ref{thm:SIE_Representation} we notice that for each fixed $t>0$ and $y\in\R$ we have almost surely
\begin{align*}
    u(t,y)-h(t,y)  & = \int_\R \int_0^t \left(b(s, x, u(s,x),v(s,x))-b(s, x, h(s,x),k(s,x))\right) G(t-s,x,y) \, \ddd s\, \ddd x \\
    &\quad +\int_\R \int_0^t \left(\sigma(s, x, u(s,x),v(s,x))-\sigma(s, x, h(s,x),k(s,x))\right)\\
    &\qquad \times G(t-s,x,y) \, W(\ddd s, \ddd x), \\
    v(t,x)-k(t,x) &= \int_0^t\left( \tilde b(s,x,u(s,x),v(s,x))-\tilde b(s,x,h(s,x),k(s,x)) \right) \, \ddd s.
\end{align*}
By the same argument as in 
\textit{Step 1}, we obtain using the Lipschitz condition \eqref{eq:Lipschitzcondition} 
that for each $T>0$ and some constant $C_{T}$
\begin{align*}
   & \|(u,v)-(h,k)\|_{t,2}^2  
\leq C_{ T} \int_0^t (t-s)^{-\frac{1}{2}}\|(u,v)-(h,k)\|_{s,2}^2    \, \ddd s
\end{align*}
for all $t\in[0,T]$.
Thus, by Gronwall's Lemma we get 
$$\|(u,v)-(h,k)\|_{t,2}^2  =0$$
and hence 
$$\PP\left((u(t,x),v(t,x))-(h(t,x),k(t,x))=0\right)=1$$
for each fixed $(t,x) \in [0,\infty[ \times \R$. 
By continuity of $u$ and $h$ we then obtain that 
\begin{align}\label{eq:pathw_unique}
    \PP(u(t,x)=h(t,x) \text{ for all } t\geq 0, x\in \R)=1.
\end{align}
By assumption, $v$ and $k$ satisfy 
\begin{align*}
 v(t,x) = v_0(x) + \int_0^t\tilde b(s,x,u(s,x),v(s,x))\,\ddd s
\end{align*}
and 
\begin{align*}
 k(t,x) = v_0(x) + \int_0^t\tilde b(s,x,h(s,x),k(s,x))\,\ddd s
\end{align*}
for all $t\ge0$ and $x\in\R$, almost surely.
Then due to Equation \eqref{eq:pathw_unique} we have on a set with probability one that for every $x \in \R$ the maps $v(\cdot, x)$ and $k(\cdot, x)$ solve the same integral equation 
(with given $u=h$) and hence coincide by uniqueness of solutions on $[0,\infty[$. Thus, we actually obtain
\begin{align*}
    \PP((u(t,x),v(t,x))=(h(t,x),k(t,x)) \text{ for all } t\geq 0, x\in \R)=1.
\end{align*}

\end{proof}

\begin{proof} [Proof of Theorem \ref{comparison}]
The proof is inspired by the proof of \cite[Thm.\ 2.3]{S94} and \cite[Thm.\ 1.1]{C17}.\\
\textit{Step 1:} We consider a regularized noise given by 
\begin{align*}
    W^x_\eps (t)= \int_0^t \int_\R \rho_\eps(x-y) \,  W (\ddd s,\ddd y)
\end{align*}
where we set $\rho_\eps(x)=G(\eps,x,0)$ for all $x \in \R$. Note that $W_\eps^x(t)$ is a local martingale with quadratic variation $\frac{1}{\sqrt{4 \eps}}t$ and is hence a constant multiple of a standard Brownian motion. This allows us to consider the following SDE
\begin{align}\label{Approxsde}
    \ddd u_\eps(t,x) &= \Delta_\eps u_\eps (t,x) +  b(t,x,u_\eps(t,x),v_\eps(t,x)) \, \ddd t  + \sigma (t,x,u_\eps(t,x)) \, \ddd W_\eps^x (t), \nonumber\\
    \ddd v_\eps(t,x)& = \tilde b(t,x,u_\eps(t,x),v_\eps(t,x)) \, \ddd t
\end{align}
where we set $$\Delta_\eps u_\eps (t,x) = \frac{1}{\eps}\left(\int_\R G(\eps,x,y) u_\eps(t,y) \, \ddd y -u_\eps(t,x)\right)=\frac{G(\eps)-I}{\eps} u_\eps (t,x).$$ For this equation we use the following facts which can also be proven by the usual Picard iteration schemes:
\begin{prop} \label{SDE}
Equation \eqref{Approxsde} has a unique (not necessarily continuous in the space variable due to the initial conditions) solution under the assumption that $b, \tilde b$ and $\sigma$ are Lipschitz in $(x,u,v)$. Moreover, this solution also satisfies the following SIE:
\begin{align} \label{SDEgleichung}
    u_\eps(t,y) &=  \int_\R   u_0(x) G_\eps (t,x,y)\, \ddd x  +\int_\R \int_0^t b(s,x,u_\eps(s,x),v_\eps(s,x)) G_\eps(t-s,x,y) \ddd s \, \ddd x  \nonumber\\
    &\quad+\int_\R \int_0^t \sigma(s,x,u_\eps(s,x)) G_\eps(t-s,x,y) \, \ddd W_\eps^x(s) \, dx,\nonumber\\
    v_\eps(t,y) &= v_0(y)+ \int_0^t \tilde b(s,x,u_\eps(s,x),v_\eps(s,x)) \, \ddd s
\end{align}
where we set $G_\eps(t,x,y) = e^{-t/\eps} \sum_{n=0}^\infty \frac{(t/\eps)^n}{n!}G(n\eps,x,y) = e^{-t/\eps}\delta_x+R_\eps(t,x,y)$.
\end{prop}
\noindent Then we may proceed as in the standard proof of comparison results (cf.\ \cite{IW86}). Note that by our assumptions and the Lipschitz condition we have that 
\begin{align}
    b(t,x,u,v) &\geq -L \abs{u},\label{conditionforcomparison1}\\
    \tilde b(t,x,u,v) &\geq -L \abs{v},\notag\\
    -b(t,x,u,v) &\geq -L\abs{1-u},\label{conditionforcomparison}\\
    -\tilde b(t,x,u,v) &\geq -L \abs{1-v}.\notag
\end{align}
Thus, approximating and localizing as in the one dimensional case \cite[Chapter VI Theorem 1.1]{IW86}\footnote{This is where we need the additional condition \eqref{conditionspositive} on $\sigma$. 
} we see that using \eqref{conditionforcomparison1}
\begin{align*}
    \E[(u_\eps(t,x))^- +(v_\eps(t,x))^-] &= - \int_0^t \E\left[ \ind_{\lbrace u_\eps(s,x) \leq 0 \rbrace} \left(\Delta_\eps u_\eps (s,x) + b(s,x,u_\eps(s,x),v_\eps(s,x))\right) \right] \,\ddd s \\
    & \qquad - \int_0^t\E\left[\ind_{\lbrace v_\eps(s,x) \leq 0 \rbrace}\tilde b(s,x,u_\eps(s,x),v_\eps(s,x))\right] \, \ddd s\\
    &\leq (L+ 1/\eps) \int_0^t \E[(u_\eps(s,x))^{-}+(v_\eps(s,x))^{-}]\, \ddd s\\
    &\qquad + \frac{1}{\eps} \int_0^t \int_\R G(\eps,x,y) \E[(u_\eps(s,y)^{-}] \, \ddd y \, \ddd s \\
    & \qquad + L \int_0^t \E[(u_\eps(s,x))^{-}+(v_\eps(s,x))^{-}] \, \ddd s
\end{align*}
where $(x)^-=\ind_{\lbrace x\leq 0\rbrace}\abs{x}$ for $x \in \R$.
Hence, we have 
$$\sup_{x \in \R}\E[(u_\eps(t,x))^- +(v_\eps(t,x))^-] \leq (2L+2/\eps) \int_0^t\sup_{x \in \R}\E[(u_\eps(s,x))^- +(v_\eps(s,x))^-] \, \ddd s.$$
An application of Gronwall's Lemma\footnote{Note that the bounded initial condition ensures via Picard iteration that $\E[(u_\eps(t,x))^- +(v_\eps(t,x))^-]$ is bounded.} will now yield that for all $(t,x) \in [0,\infty[\times\R$ 
$$\PP\left((u_\eps(t,x),v_\eps(t,x)) \in [0,\infty[^2 
\right) =1.$$
By an application of It\^{o}'s formula to $(1-u_\eps(t,x),1-v_\eps(t,x))$ and using \eqref{conditionforcomparison} instead of \eqref{conditionforcomparison1} we can proceed analogously to see that indeed for all $(t,x) \in [0,\infty[\times\R$ 
$$\PP\left((u_\eps(t,x),v_\eps(t,x)) \in [0,1]^2 
\right) =1.$$
\noindent \textit{Step 2:} We approximate $(u,v)$ by $(u_\eps,v_\eps)$. Note first that we have
\begin{align} 
    \sup_{0<\eps \leq 1} \sup_{0\leq t\leq T} \sup_{x\in \R}\E\left[\,\abs{u_\eps(t,x)}^2\right]< \infty\notag,\\
    \sup_{0<\eps \leq 1} \sup_{0\leq t\leq T} \sup_{x\in \R}\E\left[\,\abs{v_\eps(t,x)}^2\right]< \infty\notag,\\
    \sup_{0\leq t \leq T} \sup_{x\in \R}\E\left[\,\abs{u(t,x)}^2\right] <\infty\notag,\\
    \sup_{0\leq t \leq T} \sup_{x\in \R}\E\left[\,\abs{v(t,x)}^2\right] <\infty\label{bounded}
\end{align}
where the first two statements can be shown using the boundedeness of the inital condition, Lemma \ref{stupid lemma} \eqref{item:1} below and Gronwall's Lemma.
Then using the SIE representation of our solutions given by
\begin{align*} 
    u(t,y) &=  \int_\R   u_0(x) G (t,x,y)\, \ddd x  + \int_0^t\int_\R  b(s,x,u(s,x),v(s,x)) G(t-s,x,y) \ddd x\,\ddd s    \nonumber\\
    &\quad + \int_0^t\int_\R  \sigma(s,x,u(s,x)) G(t-s,x,y) \,  W(\ddd s, \ddd x),\nonumber\\
    v(t,y) &= v_0(y) + \int_0^t \tilde b(s,y,u(s,y),v(s,y)) \, \ddd s
\end{align*}
and Equation \eqref{SDEgleichung} we get by subtracting the corresponding terms
\begin{align}
    &\E[|(u_\eps(t,x),v_\eps(t,x))-(u(t,x),v(t,x))|^2]  \label{eq:giant_term}\\
    &\quad \leq C \left( \abs{\int_\R G_\eps(t,x,y)u_0(y) \, \ddd y-\int_\R G(t,x,y)u_0(y) \, \ddd y}^2\right.\nonumber\\
    &\qquad+ \E\left[\,\abs{\int_0^t e^{-(t-s)/\eps} b(s,x,u_\eps(s,x),v_\eps(s,x))\, \ddd s}^2\right]\nonumber\\
    &\qquad +\E\left[\,\abs{\int_0^t \int_\R R_\eps(t-s,x,y)(b(s,y,u_\eps(s,y),v_\eps(s,y))-b(s,y,u(s,y),v(s,y))\, \ddd s\, \ddd y}^2 \right]\nonumber\\
    &\qquad +\E\left[\,\abs{\int_0^t \int_\R (R_\eps(t-s,x,y)-G(t-s,x,y))b(s,y,u(s,y),v(s,y))\, \ddd s\, \ddd y}^2 \right]\nonumber\\
    &\qquad + \E\left[\int_0^t e^{-2(t-s)/\eps} \sigma(s,x,u_\eps(s,x))^2\, \ddd s\right] \nonumber\\
    &\qquad +\E\left[\int_0^t \int_\R\abs{ \int_\R R_\eps(t-s,x,z)(\sigma(s,z,u_\eps(s,z))-\sigma(s,z,u(s,z))) \rho_\eps(y-z)\, \ddd z}^2\, \ddd s\, \ddd y \right]\nonumber\\
    &\qquad +\E\left[\int_0^t \int_\R\abs{ \int_\R R_\eps(t-s,x,z)(\sigma(s,z,u(s,z))-\sigma(s,y,u(s,y))) \rho_\eps(y-z)\, \ddd z}^2\, \ddd s\, \ddd y \right]\nonumber\\
    &\qquad +\int_0^t \int_\R \left( \int_\R R_\eps (t-s,x,z)\rho_\eps(y-z) \, \ddd z -G(t-s,x,y)\right)^2 \E \left[\sigma(s,y,u(s,y))^2\right] \, \ddd s\, \ddd y\nonumber\\
    &\qquad \left.+ 2c \int_0^t \E\left[\left|\tilde b(s,x,u_\eps (s,x),v_\eps(s,x)) -\tilde b(s,x,u(s,x),v(s,x))\right|^2\right]\, \ddd s \right)  \nonumber\\
    &\quad=\colon C \sum_{i=1}^ {9} h_i^\eps (t,x)\nonumber
\end{align}
where we also used Fubini's theorem for Walsh's integral (cf. Theorem 5.30 in \cite{K09}). In order to proceed we recall the following elementary bounds from \cite[Lemma 6.6]{S94}  and \cite[Appendix]{C17}.
\begin{lemma} \label{stupid lemma}
For $R_\eps$ as in Proposition \ref{SDE} we have the following:
\begin{enumerate}
    \item \label{item:1} It holds for all $t>0$ and $x \in \R$ that
    $$\int_\R R_\eps (t,x,y)^2 \, \ddd y \leq \sqrt{\frac{3}{8\pi}} \frac{1}{\sqrt{t}}.$$
    \item\label{item:2} There exist constants $\delta,D >0$ such that
    $$\int_\R \abs{R_\eps(t,x,y)-G(t,x,y)} \, \ddd y \leq e^{-t/\eps} + D (\eps /t)^{1/3}$$
    for all $\eps >0$ such that $0<\eps /t \leq \delta$ and $t\geq 0, x\in \R$.
    \item\label{item:3} For all $t\geq 0$ and $x \in \R$ we have
    $$\lim_{\eps \to 0} \int_0^t \int_\R (R_\eps (s,x,y)-G(s,x,y))^2 \, \ddd s \, \ddd y =0.$$
    \item\label{item:4} 
    We have for all $x,y \in \R$ and $t\in [0,T]$, $T>0$ that
    \begin{align*}
        \E\left[\,\abs{u(t,x)-u(t,y)}^2\right] \leq C_{T,u_0} (t^{-1/2} \abs{x-y}+ \abs{x-y}).
    \end{align*}
\end{enumerate}
\end{lemma}

\begin{proof}
The proof of statements \eqref{item:1}, \eqref{item:2} and \eqref{item:3} can be found in the Appendix of \cite{C17}.\\
 For \eqref{item:4}, we write 
\begin{align*}
  u(t,y)=Z(t,y)+\int_\R G(t,y,z) u_0(z) \, \ddd z
\end{align*}
with $Z$ as in \textit{Step 2} of the proof of Theorem \ref{Lipschitzexistence}.
Then we note that by \eqref{bound_kolmogorov} (with $p=1$) and by \cite[Lemma 5.2]{M03} (using $\beta =1/2$ and $\lambda'=0$ there) we get
\begin{align*}
     \E\left[\,\abs{u(t,x)-u(t,y)}^2\right] &\leq 2\left( \E\left[\,\abs{Z(t,x)-Z(t,y)}^2\right] + \left(\int_\R |G(t,x,z)-G(t,y,z)|\,u_0(z) \, \ddd z \right)^2\right)\\
     &\le C_T|x-y| + C_{u_0}\,t^{-1/2}|x-y|
\end{align*}
for all $0\le t\le T$ and $x,y\in\R$, 
which gives the desired result.

\end{proof}

\noindent Using Lemma \ref{stupid lemma} it is now possible to show in the spirit of \cite[Theorem 1.1]{C17}, by considering each of the terms in Equation \eqref{eq:giant_term} individually, that
$$\lim_{\eps \to 0} \sup_{0\leq t \leq T} \sup_{x \in \R} \E \left[|(u_\eps(t,x),v_\eps(t,x))-(u(t,x),v(t,x))|^2\right]=0$$
for each $T>0$. 
Hence, we obtain that for all $(t,x) \in[0,\infty[\times \R$
$$\PP\left((u(t,x),v(t,x)) \in [0,1]^2 
\right) =1.$$
Since $u$ is jointly continuous on $]0,\infty[\times\R$, this of course implies
$$\PP\left(u(t,x) \in [0,1] \text{ for all } t\geq 0, x\in \R\right) =1.$$
To obtain the same result for $v$ note that by assumption 
and Equation \eqref{conditionforcomparison1}, almost surely it holds for all $t\ge0$ and $x\in\R$ that 
\begin{align*}
    v(t,x) = v_0(x) + \int_0^t \tilde b(s,x,u(s,x),v(s,x))\,\ddd s \geq v_0(x) -L \int_0^t\vert v(s,x) \vert\,\ddd s.
\end{align*}
Since $v_0(\cdot)\ge0$, by comparing path-wise to the solution of the ODE
\begin{align*}
    \partial_t f(t,x) &=-L \abs{f(t,x)}, \\
    f(0,\cdot)&\equiv 0,
\end{align*}
which is the constant zero function, we obtain path-wise for every $x \in \R$ that $v(\cdot,x) \geq 0$ on $[0,\infty[$. Hence, 
\begin{align*}
    \PP\left(v(t,x)\geq 0 \text{ for all } t\geq 0, x\in \R\right) =1.
\end{align*}
Repeating the argument for $1-v$ and using \eqref{conditionforcomparison} gives then finally
\begin{align*}
    \PP\left(v(t,x)\in [0,1] \text{ for all } t\geq 0, x\in \R\right) =1,
\end{align*}
as desired.
\end{proof}

\begin{proof} [Proof of Theorem \ref{existencereal}]

Take a sequence of Lipschitz continuous functions $(\sigma_n)_{n \in \N}$ each satisfying \eqref{conditionspositive} such that $\sigma_n \to \sigma$ uniformly on compacts as $n \to \infty$ and
\begin{align}\label{eq:UniformLinearGrowth}
    \vert\sigma_n(u)\vert \leq K(1+ \abs{u})
\end{align}
for all $u \in \R$ and $n\in\N$.
Then the maps $b,\tilde b$ and $\sigma_n$ satisfy the conditions of Theorems \ref{Lipschitzexistence} and \ref{comparison}\footnote{Note that in this context, the assumed Lipschitz continuity implies the linear growth bound \eqref{eq: lineargrowth} for $b$, $\tilde b$.} for all $n \in \N$, thus Equation \eqref{eq: SPDE_general_shortnotation} with coefficients $(b,\tilde b,\sigma_n)$ and initial conditions $u_0,v_0\in B(\R,[0,1])$ has a unique strong solution $(u_n,v_n)$ with paths 
in $C\left(]0,\infty[,C(\R,[0,1])\right)\times C\left([0,\infty[,B(\R,[0,1])\right)$.
Now, making use of assumption \eqref{eq: SPDE_general_shortnotation_Delay} to replace $v_n$ in Equation \eqref{spdeschwartz} by the corresponding quantity, 
we have that for each $\phi \in C^\infty_c (\R)$, 
almost surely it holds for all $t\ge0$ 
        \begin{align*}       
    & \langle u_n(t,\cdot), \phi \rangle=\langle u_0,\phi \rangle + \int_0^t \big\langle  u_n(s,\cdot), \tfrac{\Delta}{2} \phi\big\rangle\, \ddd s \\
    &\qquad + \int_0^t \big\langle b\left( u_n(s,\cdot),F(u_n)(s,\cdot)+H(s,v_0)(\cdot)\right), \phi \big\rangle\, \ddd s\\
   &\qquad 
    + \int_0^t\int_\R \sigma_n^2( u_n(s,x))\,\phi(x)  \, W(\ddd s, \ddd x).
       \end{align*}
Note that this is now an equation of $u_n \in C(]0,\infty [\times\R,[0,1])$ 
alone! In order to proceed, we reformulate this as the corresponding martingale problem: For each $\phi \in C^\infty_c (\R)$, the process 
\begin{align}
    M^{(n)}_t(\phi)&:= \langle u_n(t,\cdot), \phi \rangle-\langle u_0,\phi \rangle - \int_0^t \big\langle  u_n(s,\cdot), \tfrac{\Delta}{2} \phi\big\rangle\, \ddd s \nonumber\\
    &\qquad -   \int_0^t \big\langle b\left( u_n(s,\cdot),F(u_n)(s,\cdot)+H(s,v_0)(\cdot)\right), \phi \big\rangle\, \ddd s\nonumber \label{LMP}
\end{align}
is a continuous martingale with quadratic variation
\begin{equation*}
 \int_0^t \big\langle\sigma_n^2(u_n(s,\cdot)),\phi^2 \big\rangle  \, \ddd s.
\end{equation*}
We say that $ u_n$ solves $MP(\sigma_n^2,b)$. 
\\ 
\noindent Now recall from \textit{Step 2} of the proof of Theorem \ref{Lipschitzexistence} that
\begin{align*}
    Z_n (t,y) &\defeq u_n(t,y)- \hat u_0(t,y)    \\    
    & \defeq u_n(t,y)-\int_\R   u_0(x) G (t,x,y)\, \ddd x
\end{align*}
is jointly continuous on $[0,\infty[\times\R$, thus $Z_n$ takes values in $C([0,\infty[\times\R,[-1,1])$. 
Moreover, by the same calculation as there we see that \eqref{bound_kolmogorov} holds for $Z_n$ in place of $Z$ (this is where we need Equation \eqref{eq:UniformLinearGrowth}).
Hence, it follows from the Kolmogorov Chentsov Theorem \cite[Corollary 16.9]{K02} that the sequence $(Z_n)_{n \in \N}$ is tight in $C([0, \infty[\times\R,[-1,1])$ with the topology of locally uniform convergence. 
Extracting a convergent subsequence, which we continue to denote $(Z_n)_{n\in\N}$, we obtain the existence of a weak limit point $Z \in C([0, \infty[\times \R,[-1,1])$. 
Note that in order to apply the Kolmogorov Chentsov Theorem, we had to subtract the problematic quantity $\hat u_0$ which is not continuous at $t=0$.
But since $\hat u_0$ is deterministic, it follows that also $u_n=Z_n+\hat u_0\to Z+\hat u_0=:u$ weakly in $B([0,\infty[\times\R)$ as $n\to\infty$. Then clearly $u$ is almost surely $[0,1$]-valued and continuous on $]0,\infty[\times\R$, thus $u\in C(]0,\infty[\times\R,[0,1])$.

\noindent Now, using the continuous mapping theorem as in the proof of \cite[Thm.\ 21.9]{K02} we see that $ u$ actually solves the martingale problem $MP(\sigma^2,b)$, i.e.\ for each $\phi \in C_c^\infty(\R)$ the process
\begin{align*}
    M_t(\phi)& :=  \langle u(t,\cdot), \phi \rangle-\langle u_0,\phi \rangle - \int_0^t \big\langle  u(s,\cdot) , \tfrac{\Delta}{2} \phi\big\rangle\,\ddd s \\
    &\qquad - \int_0^t\big\langle b\left(u(s,\cdot),F(u)(s,\cdot)+H(s,v_0,)(\cdot)\right), \phi \big\rangle \, \ddd s
\end{align*}
is a continuous martingale with quadratic variation
\begin{equation*}
 \int_0^t \big\langle\sigma^2(u(s,\cdot)),\phi^2 \big\rangle  \, \ddd s .
\end{equation*}
\noindent Then  \cite[Thm.\ III-7]{K90} gives us the existence of a white noise process $ W$ on some filtered probability space such that
\begin{align*}
    M_t(\phi) = \int_0^t\int_\R \sigma(u(s,x))\phi(x) \, W(\ddd s, \ddd x),
\end{align*}
hence $u$ solves the SPDE 
\begin{align*}
 \langle u(t,\cdot),\phi\rangle &=\langle u_0,\phi\rangle + \int_0^t \big\langle  u(s,\cdot), \tfrac{\Delta}{2} \phi\big\rangle\, \ddd s 
+ \int_0^t\int_\R  b\left(u(s,x),F(u)(s,x)+H(s,v_0)(x)\right) \phi(x) \, \ddd x\,\ddd s  \nonumber\\
    &\qquad + \int_0^t\int_\R \sigma(u(s,x)) \phi(x) \, W(\ddd s, \ddd x).
\end{align*}

\noindent But again note that by \eqref{eq: SPDE_general_shortnotation_Delay}
$$v(t,x):=F(u)(t,x)+H(t,v_0)(x)$$
is the unique solution to the integral equation (given $u$)
\begin{align*}
    v(t,x)= v_0(x)+\int_0^t \tilde b(u(s,x),v(s,x)) \, \ddd s .
\end{align*}
Thus, $(u,v)$ is a weak solution to Equation \eqref{eq: SPDE_general_shortnotation}.

\end{proof}

\subsection*{Acknowledgments}
The authors gratefully acknowledge support from the DFG Priority Program 1590 ``Probabilistic Structures in Evolution". The third author gratefully acknowledges support from the Berlin Mathematical School (BMS).

\bibliographystyle{abbrv}
\bibliography{literature}

\end{document}